\documentclass{article}
\usepackage{proof}
\usepackage{amssymb,amsmath,amsthm}
\usepackage{times}
\usepackage{color}
\usepackage{latexsym}
\usepackage{diagrams}
\usepackage{bussproofs}
\usepackage{lscape}

\author{Jos\'e Esp\'irito Santo\\
Centro de Matem\'atica\\
Universidade do Minho\\
4710-057 Braga\\
Portugal\\
\texttt{jes@math.uminho.pt}\\ \and
Gilda Ferreira\\
DCeT, Universidade Aberta, 1269-001 Lisboa, Portugal\\
CMAFcIO, Faculdade de Ci\^{e}ncias da Universidade de Lisboa,\\ 1749-016 Lisboa, Portugal\\
\texttt{gmferreira@fc.ul.pt}
}

\title{The Russell-Prawitz embedding and the atomization of universal instantiation}
\newcommand{\lb}{\lambda}
\newcommand{\Lb}{\Lambda}
\newcommand{\pair}[2]{\langle#1,#2\rangle}
\newcommand{\ipair}[2]{\langle #2\rangle_{#1=1,2}} %----- indexed pair
\newcommand{\proj}[2]{#2#1}
\newcommand{\injnsymb}{\mathsf{in}} %-------------------------injection symb
\newcommand{\injn}[4]{\injnsymb_{#1}(#2,#3,#4)} %--------------injection
\newcommand{\ainjnsymb}{\underline{\mathtt{in}}}%--------------------------- admissible injection symbol
\newcommand{\ainjn}[4]{\ainjnsymb_{#1}(#2,#3,#4)} %--------------admissible injection
\newcommand{\rpinjnsymb}{\underline{\mathtt{IN}}}%--------------admissible injection symbol for RP
\newcommand{\rpinjn}[4]{\rpinjnsymb_{#1}(#2,#3,#4)} %--------------admissible injection for RP
\newcommand{\case}[6]{\mathsf{case}(#1,#2.#3,#4.#5,#6)}
\newcommand{\casesymb}{\mathsf{case}} %--------------- case symbol
\newcommand{\cse}[5]{\mathsf{case}(#1,#2.#3,#4.#5)} %----------- type free case
\newcommand{\acasesymb}{\underline{\mathtt{case}}} %---------------------------------admissible case symbol
\newcommand{\acase}[6]{\underline{\mathtt{case}}(#1,#2.#3,#4.#5,#6)} %--------- admissible case
\newcommand{\rpcasesymb}{\underline{\mathtt{CASE}}} %---------------------------------admissible case symbol for RP
\newcommand{\rpcase}[6]{\rpcasesymb(#1,#2.#3,#4.#5,#6)} %--------- admissible case for RP

\newcommand{\abort}[2]{\mathsf{abort}(#1,#2)}
\newcommand{\abortsymb}{\mathsf{abort}} %-------------------------------------abort symb
\newcommand{\aabortsymb}{\underline{\mathtt{abort}}} %------------------------admissible abort symb
\newcommand{\aabort}[2]{\underline{\mathtt{abort}}(#1,#2)} %---------admissible abort
\newcommand{\aabbortsymb}{\underline{\mathtt{ABORT}}} %------------------------variant of admissible abort symb
\newcommand{\aabbort}[2]{\underline{\mathtt{ABORT}}(#1,#2)} %---------variant of admissible abort
\newcommand{\rpabortsymb}{\underline{\mathtt{ABORT}}} %------------------------admissible abort symb for RP
\newcommand{\rpabort}[2]{\rpabortsymb(#1,#2)} %---------admissible abort for RP

\newcommand{\dvee}{\underline{\vee}} %--------- derived \vee to be used in Fat
\newcommand{\dperp}{\underline{\perp}} %--------- derived \perp to be used in Fat

\newcommand{\E}{\mathcal{E}}
\newcommand{\Eall}{\mathcal{E}_{\forall}} %------ E for all
\newcommand{\hole}[1]{[#1]}%-------------------- hole
\newcommand{\holec}{\_}%----------------hole contents
\newcommand{\ehole}{\hole{\holec}}%-------------------- empty hole

\newcommand{\betai}{\beta_{\supset}}%------- beta implication
\newcommand{\betac}{\beta_{\wedge}}%-------- beta conjunction
\newcommand{\betad}{\beta_{\vee}} %-------- beta disjunction
\newcommand{\betaall}{\beta_{\forall}} %-------- beta all
\newcommand{\pii}{\pi_{\supset}}%-------- pi implication
\newcommand{\pic}{\pi_{\wedge}}%-------- pi conjunction
\newcommand{\pid}{\pi_{\vee}}%-------- pi disjunction
\newcommand{\pia}{\pi_{\perp}}%-------- pi absurdity
\newcommand{\abi}{\varpi_{\supset}}%------- absurdity implication
\newcommand{\abc}{\varpi_{\wedge}}%------- absurdity conjunction
\newcommand{\abd}{\varpi_{\vee}}%------- absurdity disjunction
\newcommand{\aba}{\varpi_{\perp}}%------- absurdity absurdity
\newcommand{\ab}{\varpi}%------- commutative conversion rule for absurdity

\newcommand{\atd}{\varrho}%------- atomization disjunction
\newcommand{\ata}{\rho}%------- atomization absurdity
\newcommand{\ccd}{\varepsilon} %-------- cc disjunction
\newcommand{\cca}{\epsilon} %-------- cc absurd

\newcommand{\etai}{\eta_{\supset}}%---------- eta beta
\newcommand{\etac}{\eta_{\wedge}}%---------- eta conjunction
\newcommand{\etad}{\eta_{\vee}}%---------- eta disjunction
\newcommand{\etaall}{\eta_{\forall}}%---------- eta all

\newcommand{\ipc}{\mathbf{IPC}}
\newcommand{\fat}{{\mathbf{F}}_{\mathbf{at}}}
\newcommand{\f}{\mathbf{F}}

\newcommand{\allelimat}{\forall E_{\mathbf{at}}} %---------- all elim atomic

\newcommand{\am}[1]{#1^{\circ}} %----------------alternative map
\newcommand{\rpm}[1]{#1^{\bullet}} %----------------RP map

\newtheorem{defn}{Definition}
\newtheorem{lem}{Lemma}

\newtheorem{prop}{Proposition}
\newtheorem{thm}{Theorem}

\begin{document}

\maketitle

% remove in final version!
%\pagestyle{plain}

\begin{abstract}

Given the recent interest in the fragment of system $\f$ where universal instantiation is restricted to atomic formulas, a fragment nowadays named system $\fat$, we study directly in system $\f$ new conversions whose purpose is to enforce that restriction. We show some benefits of these new atomization conversions: (1) They help achieving strict simulation of proof reduction by means of the Russell-Prawitz embedding of $\ipc$ into system $\f$;
%(2) they are related to some ``naturality'' conversions recently proposed, but are much simper;
(2) They are not stronger than a certain ``dinaturality'' conversion known to generate a consistent equality of proofs;
(3) They provide the bridge between the Russell-Prawitz embedding and another translation, due to the authors, of $\ipc$ directly into system $\fat$; (4) They give means for explaining why the Russell-Prawitz translation achieves strict simulation whereas the translation into $\fat$ does not.
\end{abstract}


%\newpage
%\tableofcontents

%We study the translation $\ipc\to\fat$:
%\begin{enumerate}
%\item The source system is equipped with a full repertoire of reduction rules.
%\item The target system has primitive $\wedge$.
%\item Proof are handled as proof-terms (typed $\lb$-terms).
%\end{enumerate}

\section{Introduction}\label{sec:intro}

The Russell-Prawitz translation of the intuitionistic propositional calculus $\ipc$ into second-order intuitionistic propositional calculus $\mathbf{NI}^2$, the latter based on the language only containing implication, conjunction and the second-order universal quantifier, rests on the following encoding of disjunction and absurdity $A\vee B:= \forall X. ((A\supset X)\wedge(B\supset X))\supset X$ and $\perp:=\forall X. X$. This encoding is due to Prawitz but its idea goes back to Russell \cite{Prawitz65}. Under the Curry-Howard correspondence, the target of the translation can be taken to be Girard's polymorphic system $\f$ \cite{GirardLafontTaylor89}.

At the level of proofs, while the translation of the introduction rule for disjunction is straightforward, there are several alternatives for the translation of the elimination rules for the encoded connectives. The most direct one makes full use of the elimination rule for the second order quantifier. For instance, given in $\f$ a ``foreign'' formula $C$ and proofs of $A\supset C$ and $B\supset C$, it is immediate to obtain a proof of $C$ from a proof of $A\vee B$ as defined above, starting by an instantiation of the quantifier to the formula $C$. This idea is implicit in \cite{Prawitz65} - we will confirm this later on in this paper. Following \cite{Aczel2001,Ferreira2017,TranchiniPistonePetrolo2019}, we are calling Russell-Prawitz translation the translation of $\ipc$ into system $\f$ based on this translations of proofs.

There are alternative translations of $\ipc$ proofs, still employing the Russell-Prawitz translation of formulas, which make a restricted use of the elimination rule for the second order quantifier, only requiring instantiation by atomic formulas. One such translation is based on the idea of ``instantiation overflow'' \cite{Ferreira2006,FerreiraFerreira2009} - the observation that full instantiation of the quantifiers in formulas of the form $A\vee B$ or $\perp$ as above is admissible in system $\fat$ --  the restriction of system $\f$ to atomic universal instantiations. Another translation of $\ipc$ into $\fat$, always with the same translation of formulas, was proposed recently by the authors \cite{JESGFerreira2019} and is based on the observation that the elimination rules for the defined connectives are admissible in $\fat$.

There are several reasons to study embeddings of $\ipc$ into $\fat$, the foremost being that $\fat$ is a predicative fragment of $\f$. Another reason has to do with preservation of proof identities generated in $\ipc$ by commuting conversions or $\eta$-reductions: the various embeddings into $\fat$ achieve that preservation \cite{FerreiraFerreira2009,Ferreira2017,JESGFerreira2019}, while the Russell-Prawitz translation into $\f$ does not \cite{GirardLafontTaylor89,TS96,FerreiraFerreira2009,Ferreira2017,TranchiniPistonePetrolo2019}. This seems an indication that other conversion principles are missing in $\f$, besides the $\beta\eta$ ones. Indeed, a general commuting principle, expressing ``naturality'' in the categorial semantics of $\f$, was suggested in \cite{TranchiniPistonePetrolo2019}, with good results for the problem of preservation of identity of proofs.

Given the interest in system $\fat$, we study in system $\f$ other conversions whose purpose is to enforce the restriction to atomic universal instantiation, with the goal of shedding light, not only on the problem of preservation of proof identity, but also on the existence of alternative embeddings of $\ipc$ into $\f$ and $\fat$. We show that, when $\f$ is equipped with these \emph{atomization conversions}, the Russell-Prawitz translation achieves a strict simulation of proof reduction. Moreover, those conversions connect the Russell-Prawitz translation into system $\f$ and the translation into system $\fat$ introduced by the authors. First, it is easy to establish a strong relation between the two translations at the level of proofs: given an $\ipc$ proof, its translation into $\fat$ is the normal form of its translation into $\f$ w.r.t.~the atomization conversions. Second, a more laborious connection at the level of proof reduction is worked out, providing a comprehensive, detailed and clear picture of the problem of preservation of proof identity/reduction: indeed, as discussed in the final section of this paper, the key to the stronger preservation of proof \emph{reduction} is the timing of atomization.

Of course, one has to be sure that adding those atomization conversions to system $\f$ does not collapse proof identity. A similar problem was faced in \cite{TranchiniPistonePetrolo2019}, where a ``naturality'' conversion was added to system $\f$. In that paper the following argument of consistency is outlined: the equality generated by adding the ``naturality'' conversions is contained in the equality generated by adding a stronger ``dinaturality'' conversion, and the latter equality is known to have models. Here we follow the same strategy to show with considerable detail that adding atomization conversions to system $\f$ will not entail that any two terms are inter-convertible.

\textbf{Plan of the paper.} Section \ref{sec:background} recalls $\ipc$ and systems $\f$ and $\fat$. Section \ref{sec:RP} recalls the Russell-Prawitz translation, recasting it as a translation between typed $\lambda$-calculi. Section \ref{sec:atomization} introduces atomization conversions, relates them to other commuting conversions in system $\f$, and proves strict simulation by the Russell-Prawitz translation. Section \ref{sec:comparison} considers the Russell-Prawitz translation together with the embedding into $\fat$ due to the authors, relating them at the levels of proofs and proof reduction, through the atomization conversions. Section \ref{sec:final} rounds up the paper with a discussion.

\section{Background}\label{sec:background}

We present the systems we will use in the paper ($\ipc$, $\f$ and $\fat$).

\subsubsection*{System $\ipc$}

Throughout this work the interpreted system is the Intuitionistic Propositional Calculus ($\ipc$), which we now recall.

The types/formulas in $\ipc$ are given by
$$
A,B,C\,::=\,X\,|\,\perp\,|\,A\supset B\,|\,A\wedge B\,|\,A\vee B
$$
We define $\neg A:=A\supset\perp$.

The proof terms $M,N,P,Q$ are inductively generated as follows:
$$
\begin{array}{rcll}
M&::=&x&\textrm{(assumption)}\\
&|&\lb x^{A}.M\,|\,MN&\textrm{(implication)}\\
&|&\pair MN\,|\,\proj 1M\,|\,\proj 2M&\textrm{(conjunction)}\\
&|&\injn 1MAB\,|\,\injn 2NAB\,|\,\case M{x^A}P{y^B}QC&\textrm{(disjunction)}\\
&|&\abort MA&\textrm{(absurdity)}
\end{array}
$$

\noindent We work modulo $\alpha$-equivalence, in particular we assume the name of the bound variables is always appropriately chosen.

Sometimes, when clear from the context or not relevant, the type annotations in the proof terms will be omitted. This applies to type annotations in binders, or for the last type parameter in $\case M{x^A}P{y^B}QC$ and $\abort MA$\footnote{Of course, one can define a version of proof terms for $\ipc$ without the last type parameter in $\case M{x^A}P{y^B}QC$ and $\abort MA$. But, as in \cite{JESGFerreira2019}, that last type parameter makes it possible to define proof translations directly on proof terms. This happened in \emph{op. cit.} with a translation into $\fat$, and will happen in this paper with the Russell-Prawitz translation into $\f$.}. When possible and convenient, we write $\pair{P_1}{P_2}$ as $\ipair i{P_i}$.
%and we may omit the last argument of $\casesymb$.

The typing/inference rules are in Fig.~\ref{fig:typing}. A \emph{sequent} is an expression $\Gamma\vdash M:A$. An \emph{environment} $\Gamma$ is a set of \emph{declarations} $x:A$ such that each variable is declared at most one time in $\Gamma$. Given $\Gamma$ and $A$, \emph{$M$ has type $A$ in $\Gamma$} if $\Gamma\vdash M:A$ is derivable; given $\Gamma$, \emph{$M$ is typable in $\Gamma$} if, for some $A$, $M$ has type $A$ in $\Gamma$; and $M$ is \emph{typable} if $M$ is typable in some $\Gamma$.

%--------------------------------
\begin{figure}[t]\caption{Typing/inference rules of $\ipc$}\label{fig:typing}
$$
\begin{array}{c}
\infer[Ass]{\Gamma,x:A\vdash x:A}{}\\ \\
\infer[\supset\! I]{\Gamma\vdash\lb x^{A}.M:A\supset B}{\Gamma,x:A\vdash M:B}\qquad
\infer[\supset\! E]{\Gamma\vdash MN:B}{\Gamma\vdash M:A\supset B&\Gamma\vdash N:A}\\ \\
\infer[\wedge I]{\Gamma\vdash \pair MN:A\wedge B}{\Gamma\vdash
M:A&\Gamma\vdash N:B}\qquad
\infer[\wedge E1]{\Gamma\vdash \proj 1M:A}{\Gamma\vdash M:A\wedge B}\qquad\infer[\wedge E2]{\Gamma\vdash \proj 2M:B}{\Gamma\vdash M:A\wedge B}\\ \\
\infer[\vee I1]{\Gamma\vdash\injn 1MAB:A\vee B}{\Gamma\vdash M:A}\qquad
\infer[\vee I2]{\Gamma\vdash\injn 2NAB:A\vee B}{\Gamma\vdash N:B}\\ \\
\infer[\vee E]{\Gamma\vdash\case M{x^A}P{y^B}QC:C}{\Gamma\vdash
M:A\vee B&\Gamma,x:A\vdash P:C&\Gamma,y:B\vdash Q:C}\\ \\
\infer[\perp\! E]{\Gamma\vdash\abort MA:A}{\Gamma\vdash M:\perp}
\end{array}
$$
\end{figure}
%--------------------------------

For the purpose of discussing some reduction rules and defining the translation
of proof terms, it is convenient to arrange the syntax of the system in a
different way:
$$
\begin{array}{rrcl}
\textrm{(Terms)}&M&::=&V\,|\,\E\hole M\\
\textrm{(Values)}&V&::=&x\,|\lb x.M\,|\,\pair MN\,|\,\injn 1MAB\,|\,\injn
2NAB\\
\textrm{(Elim. contexts)}&\E&::=&\ehole N\,|\,\proj
1{\ehole}\,|\,\proj 2{\ehole}\\
&& | & \case {\ehole}{x}P{y}QC\,|\,\abort{\ehole}A
\end{array}
$$
\noindent A \emph{value} $V$ ranges over terms representing
assumptions or introduction inferences. $\E$ stands for an
\emph{elimination context}, which is a term representing an
elimination inference, but with a ``hole'' in the position of the
main premiss. $\E\hole M$ denotes the term resulting from filling
the hole of $\E$ with $M$.

%$$
%\E::=\ehole N\,|\,\proj 1{\ehole}\,|\,\proj 2{\ehole}\,|\,\cse
%{\ehole}{x^A}P{y^B}Q
%$$
%\noindent $\E\hole M$ denotes the term resulting from filling the
%hole of $\E$ with $M$.

In Fig.~\ref{fig:typing-contexts} one finds the typing rules for
elimination contexts. In a sequent $\Gamma|A\vdash\E:B$, the type $A$ is the type of the hole of $\E$ and $B$ is
the type of the term obtained by filling the hole of $\E$ with a
term of type $A$. If the sequent $\Gamma|A\vdash\E:B$ is derivable, we say \emph{$\E$ has type $B$ and hole of type $A$ in $\Gamma$}.
%------------------------
\begin{figure}[t]\caption{Typing rules for elimination contexts}\label{fig:typing-contexts}
$$
\begin{array}{c}
\infer{\Gamma|\perp\,\,\vdash\abort {\ehole}A:A}{}\qquad\infer[(i=1,2)]{\Gamma|A_1\wedge A_2\vdash\proj
i{\ehole}:A_i}{}\\ \\\infer{\Gamma|A\supset B\vdash\ehole N:B}{\Gamma\vdash
N:A}\qquad
\infer{\Gamma|A\vee B\vdash\case{\ehole}xPyQC:C}{\Gamma,x:A\vdash
P:C&\Gamma,y:B\vdash Q:C}\\ \\
\infer{\Gamma\vdash\E\hole M:B}{\Gamma\vdash M:A&\Gamma|A\vdash\E:B}
\end{array}
$$
\end{figure}
%----------------------------------

The reduction rules are given in Fig.~\ref{fig:reduction-rules}. We let $\beta:=\betai\cup\betac\cup\betad$ and similarly for $\eta$; and we let $\pi:=\pii\cup\pic\cup\pid\cup\pia$ and similarly for $\ab$.
%The detour conversion rules make use of ordinary substitution $[N/x]M$. The commuting conversion rules make use of a specific
%organization of the definition of elimination contexts: %$\E::=\Ei\,|\,\Ec\,|\,\Ed\,|\,\Ea$, where
%$$
%\begin{array}{rclcrcl}
%\Ei&::=&\ehole N&\qquad&\Ed&::=&\cse{\ehole}{x}P{y}Q\\
%\Ec&::=&\proj 1{\ehole}\,|\,\proj 2{\ehole}&\qquad&\Ea&::=&\abrt{\ehole}
%\end{array}
%$$
%$$
%\begin{array}{rclcrclcrcl}
%\E&::=&\Ei\,|\,\Ec\,|\,\Ed\,|\,\Ea&\,\,&\Ei&::=&\ehole N&\,\,&\Ed&::=&\cse{\ehole}{x}P{y}Q\\
%&&&&\Ec&::=&\proj 1{\ehole}\,|\,\proj 2{\ehole}&\,\,&\Ea&::=&\abrt{\ehole}
%\end{array}
%$$

%--------------------------------
\begin{figure}[t]\caption{Reduction rules}\label{fig:reduction-rules}
%\vspace{.25cm}
Detour conversion rules:
$$
\begin{array}{rrcll}
(\betai)&(\lb x.M)N&\to&[N/x]M&\\
(\betac)&\proj i{\pair{M_1}{M_2}}&\to&M_i&\textrm{ ($i=1,2$)}\\
(\betad)&\cse{\injn iM{A_1}{A_2}}{x_1^{A_1}}{P_1}{x_2^{A_2}}{P_2}&\to&[M/x_i]P_i&\textrm{ ($i=1,2$)}
\end{array}
$$
Commutative conversion rules for disjunction (in the 2nd rule, $i\in\{1,2\}$):
$$
\begin{array}{rrcl}
(\pii) & (\case{M}{x}{P}{y}{Q}{C\supset D})N & \to & \case{M}{x}{PN}{y}{QN}{D}\\
(\pic) & \proj i{(\case{M}{x}{P}{y}{Q}{C_1\wedge C_2})} & \to & \case{M}{x}{\proj iP}{y}{\proj iQ}{C_i}\\ %\qquad(i=1,2)\\
(\pid) & \\
\multicolumn{2}{c}{\case{\case{M}{x'}{P'}{y'}{Q'}{C\vee D}}{x^C}{P}{y^D}{Q}{E}}  & \to & \\ %\case M{x}{\case{P}{x'}{P'}{y'}{Q'}{E}}{y}{\case{Q}{x'}{P'}{y'}{Q'}{E}}{C\vee D}\\
        \multicolumn{4}{c}{\case M{x'}{\case{P'}{x^C}{P}{y^D}{Q}{E}}{y'}{\case{Q'}{x^C}{P}{y^D}{Q}{E}}{E}}\\
(\pia) & \abort{\case{M}{x}{P}{y}{Q}{\perp}}{C} & \to &\\
        \multicolumn{4}{c} {\case{M}{x}{\abort P C}{y}{\abort Q C}{C}}
\end{array}
$$

Commutative conversion rules for absurdity (in the 2nd rule, $i\in\{1,2\}$):%Absurdity rules
$$
\begin{array}{rrcl}
(\abi) & (\abort M{C\supset D})N & \to & \abort MD\\
(\abc) & (\abort M{C_1\wedge C_2})i & \to & \abort M{C_i}\\ % \qquad(i=1,2)\\
(\abd) & \case{\abort M{C \vee D}}{x^C}P{y^D}Q{E} & \to & \abort ME \\
(\aba) & \abort{\abort M{\perp}}C & \to & \abort MC
\end{array}
$$
$\eta$-rules:
$$
\begin{array}{rrcll}
(\etai)&\lb x.Mx&\to&M&(x\notin M)\\
(\etac)&\pair{\proj 1M}{\proj 2M}&\to&M&\\
(\etad)&\case M{x^A}{\injn 1x{A}{B}}{y^B}{\injn 2yAB}{A\vee B}&\to&M&
\end{array}
$$
\end{figure}
%--------------------------------

Given a reduction rule $R$ of $\ipc$, we employ the usual notations concerning reduction relations generated by $R$: the compatible closure\footnote{A relation $R$ on the proof terms is \emph{compatible} if it is compatible with each proof-term constructor. For instance, $R$ is compatible with the application constructor if $M\,R\,M'$ implies $(MN)\,R\,(M'N)$, and $N\,R\,N'$ implies $(MN)\,R\,(MN')$.} of $R$ is denoted $\to_R$; and $\to_R^+$, $\to_R^*$, $=_R$ denote respectively the transitive closure, the reflexive-transitive closure, and the reflexive-symmetric-transitive closure of $\to_R$. If $R=R_1\cup R_2$, then we may omit ``$\cup$'' in our notation and write $\to_{R_1R_2}$, etc. The same notations apply to systems $\f$ and $\fat$ below.

%--------------------------------------
%\begin{prop}
%Let $R$ be a reduction rule of $\ipc$. If $\Gamma\vdash M:A$ is derivable and $M\to_RN$ then $\Gamma\vdash N:A$ is derivable.
%\end{prop}
%--------------------------------------
%\begin{proof}
%By induction on $M\to_RN$.
%\end{proof}

%This is the ``subject reduction'' property, which states that reduction preserves types. The proof shows how to obtain a derivation of $\Gamma\vdash N:A$ from a given derivation of $\Gamma\vdash M:A$ when $M\to_RN$. The interesting case is the base case, corresponding to the reduction rule itself: the derivation of $\Gamma\vdash N:A$ is obtained by the familiar procedures that eliminate a maximal formula, or shorten a segment, etc. We may see the proof of this proposition as defining the proof transformation induced by the reduction rule $R$.

For every reduction rule $R$ of $\ipc$, $\to_R$ has the \emph{subject reduction} property, that is: if $M\to_R N$ and $M$ has type $A$ in $\Gamma$, then $N$ has type $A$ in $\Gamma$.

Notice $\pi$ is contained in the following reduction rule:
%$$
\begin{equation}\label{commutative-conversion-with-context}
\E\hole{\case M{x^A}P{y^B}QC}\to\case M{x^A}{\E\hole P}{y^B}{\E\hole Q}D %\qquad(*)
\end{equation}
%$$
But the rules are not equivalent, as witnessed by the fact that the latter rule generates a reduction relation $\to$ that does not enjoy subject reduction. This is caused by the fact that types $C$ and $D$ are unconstrained in (\ref{commutative-conversion-with-context}), and so the connections that the various $\pi$-rules establish between the types $C$ and $D$ are not imposed in (\ref{commutative-conversion-with-context}). For instance, in rule $\pii$, $C$ is an implication $E\supset F$, say, and $D$ is $F$. But, in (\ref{commutative-conversion-with-context}), when $\E=\ehole N$, even if we know that the redex has type $F$ in $\Gamma$, and hence $C=E\supset F$, for some $E$, we have no reason to infer that $D=F$, and so we cannot conclude that the contractum has type $F$ in $\Gamma$.

If we wanted to constrain (\ref{commutative-conversion-with-context}) to become equal to $\pi$, we would define it only for certain pairs $(\E,C)$, with $D=D(\E,C)$ determined by $(\E,C)$ as follows:
\begin{itemize}
\item  $D(\ehole N, E\supset F)=F$;
\item $D(\ehole i,C_1\wedge C_2)=C_i$, for $i=1,2$;
\item $D(\case{\ehole}{x^A}{P}{y^B}{Q}{E},A\vee B)=E$; % and
\item $D(\abort{\ehole}{E},\perp)=E$.\footnote{Curiously, if the proof terms had been defined without the last type parameter in $\casesymb$- and $\abortsymb$-expressions, then the version of (\ref{commutative-conversion-with-context}) without $C$ and $D$ would determine a reduction relation $\to$ enjoying subject reduction. }
\end{itemize}

Similar remarks apply to $\ab$ versus $\E\hole{\abort MC}\to\abort MD$.

%++++++++++++++++++++++++++++++++++++

%\vspace{.5cm}

\subsubsection*{System $\f$}

We give a precise definition of the polymorphic system $\f$ by saying what changes relatively to $\ipc$ (for an introduction to system $\f$, see \cite{GirardLafontTaylor89}). In the spirit of the Curry-Howard correspondence, we sometimes refer to $\f$ as the natural deduction system $\mathbf{NI}^2$.

Regarding formulas, $\perp$ and $A\vee B$ are dropped, and the new form $\forall X.A$ is adopted (hence conjunction is taken as primitive in system $\f$). The quantifier $\forall X$ binds free occurrences of $X$, inducing the obvious concept of free occurrence of a type variable in a type. Concerning $\alpha$-equivalence, we deal with type variables as we deal with term variables, relying on silent $\alpha$-renaming. We write $X\notin A$ to say that $X$ does not occur free in $A$; given the silent $\alpha$-renaming in $A$, we may assume $X$ does not occur bound in $A$ either. Another novelty, distinctive of system $\f$, is type substitution in types, $[B/X]A$, meaning: substitution in $A$ of each free occurrence of $X$ by the type $B$.

Regarding proof terms, the constructions relative to $\perp$ and $A\vee B$ are dropped, and the new forms $\Lambda X.M$ and $MB$, with $B$ a type, are added. The latter gives rise to a new form of evaluation contexts: $\E\,::=\, \ehole B$. %$\Eall\,::=\, \ehole B$. %Types occur in proof terms, not only via $MX$, but also via the type annotations in $\lambda$-abstractions; in particular, this is how type variables may occur free in proof terms. Accordingly, there is the operation of type substitution in proof terms, denoted $[Y/X]M$, defined by recursion on $M$: the critical equations are $[Y/X](MX)=([Y/X]M)Y$ and $[Y/X](\lb x^A.M)=\lb x^{[Y/X]A}.[Y/X]M$. Again, we write $X\notin M$ to say that $X$ does not occur free in $M$, which is the same to say $X$ does not occur at all in $M$, due to the assumed $\alpha$-renaming of type variables.

%Contexts $\E$ will be of no use in $\fat$.

Regarding typing rules, those relative to $\perp$ and $A\vee B$ are dropped, and two rules relative to $\forall X.A$ are adopted:
$$
\infer[\forall I]{\Gamma\vdash \Lambda X.M:\forall X.A}{\Gamma\vdash M:A}\qquad\infer[\forall E]{\Gamma\vdash MB:[B/X]A}{\Gamma\vdash M:\forall X.A}
$$
where the proviso for $\forall I$ is: %$X$ does not occur free in some type in $\Gamma$.
$X$ occurs free in no type in $\Gamma$. Due to rule $\forall E$, the construction $MB$ is called \emph{universal instantiation}.

The new form of elimination contexts $\Eall$ is typed with:
$$
\infer{\Gamma|\forall X.A\vdash\ehole B:[B/X]A}{}
$$

Regarding reduction rules, we drop commuting conversion rules (since they are relative to $\vee$ and $\perp$). What remains are the $\beta$ and $\eta$-rules (but we drop those relative to disjunction). For $\forall$, these are:
$$
\begin{array}{rrcll}
(\betaall)&(\Lb X.M)B&\to&[B/X]M&\\
(\etaall)&\Lb X.MX&\to&M&(X\notin M)\\
\end{array}
$$

We let $\beta:=\betai\cup\betac\cup\betaall$. Similarly for $\eta$.
%$$
%\begin{array}{rcl}
%\beta&:=&\betai\cup\betac\cup\betaall\\
%\pi&:=&\pii\cup\pic\cup\pid\cup\pia\\
%\eta&:=&\etad
%\end{array}
%$$

\subsubsection*{System $\fat$}

The atomic polymorphic system $\fat$, is the fragment of system $\f$ induced by restricting to atomic instances the elimination inference rule for $\forall$, and the corresponding proof term constructor.

Thus the types/formulas of $\fat$ are exactly the types of $\f$  with type substitution in types only required in the atomic form $[Y/X]A$ and the proof terms of $\fat$ only differ from the proof terms of $\f$ in the construction relatively to $\forall E$ being $MB$ replaced by $MX$ which gives rise to $\Eall\,::=\, \ehole X$.

%Types occur in proof terms, not only via $MX$, but also via the type annotations in $\lambda$-abstractions; in particular, this is how type variables may occur free in proof terms. Accordingly, there is the operation of type substitution in proof terms, denoted $[Y/X]M$, defined by recursion on $M$: the critical equations are $[Y/X](MX)=([Y/X]M)Y$ and $[Y/X](\lb x^A.M)=\lb x^{[Y/X]A}.[Y/X]M$. Again, we write $X\notin M$ to say that $X$ does not occur free in $M$, which is the same to say $X$ does not occur at all in $M$, due to the assumed $\alpha$-renaming of type variables.

%Contexts $\E$ will be of no use in $\fat$.

The typing rules relatively to $\forall X.A$ are:
$$
\infer[\forall I]{\Gamma\vdash \Lambda X.M:\forall X.A}{\Gamma\vdash M:A}\qquad\infer[\allelimat]{\Gamma\vdash MY:[Y/X]A}{\Gamma\vdash M:\forall X.A}
$$
where the proviso for $\forall I$ is: $X$ occurs free in no type in $\Gamma$. The new form of elimination contexts $\Eall$ is typed with:
$$
\infer{\Gamma|\forall X.A\vdash\ehole Y:[Y/X]A}{}
$$

The reduction rules for $\forall$ are:
$$
\begin{array}{rrcll}
(\betaall)&(\Lb X.M)Y&\to&[Y/X]M&\\
(\etaall)&\Lb X.MX&\to&M&(X\notin M)\\
\end{array}
$$

%We let $\beta:=\betai\cup\betac\cup\betaall$. Similarly for $\eta$.
%$$
%\begin{array}{rcl}
%\beta&:=&\betai\cup\betac\cup\betaall\\
%\pi&:=&\pii\cup\pic\cup\pid\cup\pia\\
%\eta&:=&\etad
%\end{array}
%$$

\section{The Russell-Prawitz embedding}\label{sec:RP}

We recall the Russell-Prawitz translation of $\ipc$ into $\f$. The treatment is by means of proof terms rather than derivations. In this section we just define the translation, observe type soundness, and revisit \cite{Prawitz65} to justify the designation ``Russell-Prawitz''. The matter of preservation of proof reduction is postponed to the next section.

\begin{defn} In $\f$ (and in $\fat$):
\begin{enumerate}
\item $A\dvee B:=\forall X.((A\supset X)\wedge(B\supset X))\supset X$, with $X\notin A,B$.
\item $\dperp:=\forall X.X$.
\end{enumerate}
\end{defn}

We define the Russell-Prawitz translation of formulas. Using the abbreviations just introduced, the definition can be given in a homomorphic fashion:

$$
\begin{array}{rcl}
\rpm X&=&X\\
\rpm \perp&=&\dperp\\
\rpm{(A\supset B)}&=&\rpm A\supset\rpm B\\
\rpm{(A\wedge B)}&=&\rpm A\wedge\rpm B\\
\rpm{(A\vee B)}&=&\rpm A\dvee \rpm B
\end{array}
$$

The translation of proof terms will rely on the following definition:

%-------------------------------------------------
\begin{defn}\label{def:admissible-constructions-F} In $\f$:
\begin{enumerate}
\item Given $M,A,B$, given $i\in\{1,2\}$, we define
$$\rpinjn iMAB:=\Lb X.\lb w^{(A\supset X)\wedge(B\supset X)}.\proj i w M\enspace,$$
where the bound variable $X$ is chosen so that $X\notin M,A,B$.
\item Given $M,P,Q,A,B,C$, we define
$$\rpcase M{x^A}P{y^B}QC:=MC\pair{\lb x^A. P}{\lb y^B. Q}$$
\item Given $M,A$, we define $\rpabort MA:=MA$
\end{enumerate}
\end{defn}
%-------------------------------------------------

It is straightforward to see that the typing rules in Fig.~\ref{fig:admissible-typing-rules-F} - that is, the inference rules for disjunction and absurdity - are derivable in $\f$.

\begin{figure}[t]\caption{Derivable typing rules of $\f$}\label{fig:admissible-typing-rules-F}
$$
\begin{array}{c}
\infer[(i=1,2)]{\Gamma\vdash\rpinjn iM{A_1}{A_2}:A_1\dvee A_2}{\Gamma\vdash M:A_i}\\ \\
\infer{\Gamma\vdash\rpcase M{x^A}P{y^B}QC:C}{\Gamma\vdash M:A\dvee B&\Gamma,x:A\vdash P:C&\Gamma,y:B\vdash Q:C}\\ \\
\infer{\Gamma\vdash\rpabort MC:C}{\Gamma\vdash M:\dperp}
\end{array}
$$
\end{figure}

The following result is also straightforward:
%------------------------------------------------
\begin{lem}\label{lem:compatibility-F}
Let $R$ be a relation compatible in the proof-terms of $\f$. Then the compatibility rules in Fig. \ref{fig:compatibility-rules-F} are derivable in $\f$.
\end{lem}

%--------------------------------------------------
\begin{figure}[t]\caption{Derivable compatibility rules of $\f$}\label{fig:compatibility-rules-F}
$$
\begin{array}{c}
\infer{\rpinjn iMAB\,R\,\rpinjn i{M'}AB}{M\,R\,M'}\\ \\
\infer{\rpcase MxPyQC\,R\,\rpcase {M'}xPyQC}{M\,R\,M'}\\ \\
\infer{\rpcase MxPyQC\,R\,\rpcase Mx{P'}yQC}{P\,R\,P'}\\ \\
\infer{\rpcase MxPyQC\,R\,\rpcase MxPy{Q'}C}{Q\,R\,Q'}\\ \\
\infer{\rpabort MC\,R\,\rpabort {M'}C}{M\,R\,M'}
\end{array}
$$
\end{figure}
%--------------------------------------------------

Due to Definition \ref{def:admissible-constructions-F}, the translation of proof terms can be given in a purely homomorphic fashion:
%----------------------------
\begin{defn}
Given $M\in\ipc$, $\rpm M$ is defined by recursion on $M$ as in Fig. \ref{fig:translation-F}.
\end{defn}

%------------------------
\begin{figure}[t]\caption{The Russell-Prawitz translation of proof expressions}\label{fig:translation-F}
$$
\begin{array}{rcll}
\rpm x&=&x&\\
\rpm{(\lb x^A.M)}&=&\lb x^{\rpm A}.\rpm M&\\
\rpm{\pair MN}&=&\pair{\rpm M}{\rpm N}&\\
\rpm{(\injn iMAB)}&=&\rpinjn i{\rpm M}{\rpm A}{\rpm B}&\textrm{($i=1,2$)}\\
%\am{(\injn iMAB)}&=&\Lb X.\lb w^{(\am A\supset X)\wedge(\am B\supset X)}.\pi_i(w_i)\am M&\textrm{($i=1,2$)}\\
%\rpm{(\Es\hole M)}&=&\rpm{\Es}\hole{\rpm M}&(\connective=\supset,\wedge)\\
\rpm{(\E\hole M)}&=&\rpm{\E}\hole{\rpm M}&\\
\rpm{(\case M{x^A}P{y^B}QC)}&=&\rpcase{\rpm M}{x^{\rpm A}}{\rpm P}{y^{\rpm B}}{\rpm Q}{\rpm C}&\\
\rpm{(\abort MA)}&=&\rpabort{\rpm M}{\rpm A}&\\
&&&\\
\rpm{(\ehole N)}&=&\ehole\rpm N\\
\rpm{(\proj i{\ehole})}&=&\proj i{\ehole}
\end{array}
$$
\end{figure}
%----------------------------------

\noindent Notice that $\rpm{(MN)}=\rpm M\rpm N$ and $\rpm{(\proj iM)}=\proj i{\rpm M}$.

Observe the use of the type information provided by the last argument of $\casesymb$- and $\abortsymb$-expressions: from $C$ in $\case M{x^A}P{y^B}QC$ we determine the argument $\rpm C$ required by $\rpcasesymb$; from $A$ in $\abort MA$ we determine the argument $\rpm{A}$ for $\rpabortsymb$.\footnote{If the proof terms of $\ipc$ had been defined without the last type parameter in $\casesymb$- and $\abortsymb$-expressions, instead of a translation of proof terms, we would have a translation of typing derivations.} % in $\ipc$.}

\begin{prop}[Type soundness] If $\Gamma\vdash M:A$ in $\ipc$, then $\rpm{\Gamma}\vdash\rpm M:\rpm A$ in $\f$.
\end{prop}
%------------------------------------------------

The easy proof of this proposition determines a transformation of derivations in $\ipc$ into derivations in $\f$, a transformation defined by recursion on the given derivation in $\ipc$, based on the admissibility in $\f$ of each inference rule of $\ipc$.

We now argue that such a transformation is already implicit in \cite{Prawitz65}. For this discussion, let disjunction be a primitive connective of $\mathbf{NI}^2$ (here, contrary to \cite{Prawitz65}, we will ignore the second-order existential quantifier). This corresponds to extending system $\f$ with the type former $A\vee B$ and the constructions $\injn iMAB$ ($i=1,2$) and $\case M{x^A}P{y^B}QC$. Prawitz \cite{Prawitz65} shows that in $\mathbf{NI}^2$ the connectives $\wedge$, $\vee$ and $\perp$ are definable operations. For instance, in the case of disjunction, this means that $(A_1\vee A_2)\supset (A_1\dvee A_2)$ and $(A_1\dvee A_2)\supset (A_1\vee A_2)$ are theorems of $\mathbf{NI}^2$. The proof in \cite{Prawitz65}, in terms of the extended system $\f$, amounts to the following derivable sequents:
\begin{equation}\label{eq:Prawitz1}
y:A_1\vee A_2\vdash \Lambda X\lambda w^{(A_1\supset X)\wedge(A_2\supset X)}.\case{y}{x_1}{w1x_1}{x_2}{x_2}{w2x_2}:A_1\dvee A_2
\end{equation}
\begin{equation}\label{eq:Prawitz2}
z:A_1\dvee A_2\vdash z(A_1\vee A_2)\pair{\lb x_1.\injn 1{x_1}{A_1}{A_2}}{\lb x_2.\injn 2{x_2}{A_1}{A_2}}:A_1\vee A_2
\end{equation}
This is very close to show the admissibility of the introduction and elimination rules for $A_1\vee A_2$ (the first two rules of Fig.~\ref{fig:admissible-typing-rules-F}). Given $\Gamma\vdash M:A_i$, from (\ref{eq:Prawitz1}) we get
$$
\Gamma\vdash \Lambda X\lambda w^{(A_1\supset X)\wedge(A_2\supset X)}.\case{\injn iM{A_1}{A_2}}{x_1}{w1x_1}{x_2}{x_2}{w2x_2}:A_1\dvee A_2
$$
Applying $\betad$, the term reduces to $\Lambda X\lambda w^{(A_1\supset X)\wedge(A_2\supset X)}.wiM$. %, which we might abbreviate as $\rpinjn iM{A_1}{A_2}$, had we adopted in the extended system $\f$ the abbreviations in Definition \ref{def:admissible-constructions-F}.
On the other hand, given $\Gamma,x_i:A_i\vdash P_i:C$, for $i=1,2$, a variation of (\ref{eq:Prawitz2}) gives
$$
\Gamma,z:A_1\dvee A_2\vdash zC\pair{\lb x_1.P_1}{\lb x_2.P_2}:C
$$
So, if we are further given $\Gamma\vdash M:A_1\dvee A_2$, we obtain
$$
\Gamma\vdash MC\pair{\lb x_1.P_1}{\lb x_2.P_2}:C
$$
%The conclusion is

\section{Atomization of universal instantiation}\label{sec:atomization}

In this section we add to system $\f$ extra conversions $\atd$ and $\ata$ which promote the atomization of universal instantiation.
%found in image of $\rpm{(\cdot)}$.
We show typable terms have unique ``atomic'' normal forms.
We also propose new conversions $\ccd$ and $\cca$, which postulate the commuting principles for the derived connectives of disjunction and absurdity, and which are simple variants of a general commuting principle introduced in \cite{TranchiniPistonePetrolo2019}\footnote{Notice, however, that here, contrary to \cite{TranchiniPistonePetrolo2019}, in the formulation of the commuting principles $\ccd$ and $\cca$, we do not constraint ourselves to formulas obeying certain restrictions in the polarity of the occurrences of type variables.}. The latter principle adds ``naturality'' to natural deduction, according to \cite{TranchiniPistonePetrolo2019}. Here we work out in detail the relationship of $\ccd$ and $\cca$ with a more general ``dinaturality'' principle. The relationship between $\atd$, $\ata$ and $\ccd$, $\cca$ is also worked put. Since it has been proved that adding the dinaturality principle to system $\f$ does not make the system inconsistent \cite{BainbridgeFreydScedrovScottTCS90}, the same follows about adding $\ccd$ and $\cca$. We profit from the relationship among all the new conversions to prove that extending system $\f$ with $\atd$ and $\ata$ does not bring inconsistency. We also show the simulation theorem for the Russell-Prawitz translation, which makes use of the atomization conversions.

%==================================
\subsection{New conversions for system $\f$}

The $\atd$- and $\ata$-redexes are terms of the form $\rpcase MxPyQC$ and $\rpabort MC$, respectively, where $C$ is not atomic. Therefore, such redexes include a universal instantiation with a non-atomic formula $C$. The common purpose of each of the $\atd$- and $\ata$-conversion rules is to replace one such instantiation by another with a sub-formula of $C$. Since there is a common purpose, the two conversion rules are denoted with symbols ``$\atd$'' and ``$\ata$'' which are variant of each other.
%The typographic distinction between ``$\atd$'' and ``$\ata$'' is intended.

%----------------------
\begin{defn}[Atomization conversion rules in $\f$]\label{def:atomization-rules}
\begin{enumerate}
\item A $\atd$-redex is a term of the form $M C \pair{\lb x^A.P}{\lb y^B.Q}$, where $C$ is not atomic. There are three $\atd$-conversion rules for this redex, depending on the form of $C$:
$$
\begin{array}{rcl}
M(C_1\supset C_2)\pair{\lb x^{A}.P}{\lb y^{B}.Q}&\to&\lb z^{C_1}.M C_2 \pair{\lb x^{A}.Pz}{\lb y^{B}.Qz}\\
M(C_1\wedge C_2)\pair{\lb x^{A}.P}{\lb y^{B}.Q}&\to&\ipair i{ MC_i\pair{\lb x^{A}.Pi}{\lb y^{B}.Qi}}\\
M(\forall Y.D)\pair{\lb x^{A}.P}{\lb y^{B}.Q}&\to&\Lb Y. MD\pair{\lb x^{A}.PY}{\lb y^{B}.QY}
\end{array}
$$
\noindent where $z\neq x$, $z\neq y$, $z\notin P,Q,M$; and $Y\notin P,Q,M,A,B$.

\item A $\ata$-redex is a term of the form $M C$, where $C$ is not atomic. There are three $\ata$-conversion rules for this redex, depending on the form of $C$:
$$
\begin{array}{rcl}
M(C_1\supset C_2)&\to&\lb z^{C_1}.M C_2\\
M(C_1\wedge C_2)&\to&\ipair i{ MC_i}\\
M(\forall Y.D)&\to&\Lb Y. MD
\end{array}
$$
\noindent where $z\notin M$; and $Y\notin M$.
\end{enumerate}
\end{defn}
For now, a $\atd$-redex always contains a $\ata$-redex. Typing constraints will later forbid this situation. However, typing considerations for these rules and the discussion of subject reduction are postponed to Subsection \ref{subsec:properties-typed-atomization}.

We now introduce a variant of the atomization conversion $\atd$.

%-------------------------------
\begin{defn} \label{def:delta} The $\delta$-conversion rules are as follows:
$$
\begin{array}{rcl}
M(C_1\supset C_2)\pair{\lb x^{A}\lb z^{C_1}.P}{\lb y^{B}\lb z^{C_1}.Q}&\to&\lb z^{C_1}.MC_2\pair{\lb x^{A}.P}{\lb y^{B}.Q}\\
M(C_1\wedge C_2)\pair{\lb x^{A}.\ipair i{P_i}}{\lb y^{B}.\ipair i{Q_i}}&\to&\ipair i{ MC_i\pair{\lb x^{A}.P_i}{\lb y^{B}.Q_i}}\\
M(\forall Y.D)\pair{\lb x^{A}\Lb Y.P}{\lb y^{B}\Lb Y.Q}&\to&\Lb Y. MD\pair{\lb x^{A}.P}{\lb y^{B}.Q}
\end{array}
$$
%\noindent where, in each case, the redex is typable in $\Gamma$ and  $M$ has type $A\dvee B$ in $\Gamma$.
%$\forall X.((A\supset X)\wedge (B\supset X))\supset X$.
\end{defn}
%------------------------------

A $\delta$-redex is a particular form of $\atd$-redex which, if reduced by $\atd$, generates two $\beta$-redexes (actually four, in the case of conjunction). If these are reduced away immediately, we obtain the effect of $\delta$-reduction. Conversely, a $\atd$-redex, if $\eta$-expanded, can be reduced with $\delta$ instead of $\atd$. So, $\delta$ and $\atd$ are related via $\beta\eta$-conversions, as the next result shows.

%It also emphasizes that the $\atd$-conversion rules should be taken as primitive instead of the $\delta$-conversion rules.

%---------------------------
\begin{prop}[Variants of atomization]\label{prop:variants-of-atomization} Let $M,N\in \f$.
\begin{enumerate}
\item If $M\to_{\delta}N$ then $M\to^+_{\atd\beta}N$.
\item If $M\to_{\atd}N$ then $M=_{\delta\eta}N$.
\end{enumerate}
\end{prop}
%----------------------------------------------
\begin{proof} We have to do an induction on $M\to_{\delta}N$ and another on $M\to_{\atd}N$. In both proofs, the inductive cases are routine because the relations $\to^+_{\atd\beta}$ and $=_{\delta\eta}$ are compatible. In each proof there are 3 base cases. We just illustrate with two base cases for the first assertion and one for the second.
$$
\begin{array}{cl}
&M(C\supset D)\pair{\lb x^A\lb z^C.P}{\lb y^B\lb z^C.Q}\\
\to_{\atd}&\lb w^C.MD\pair{\lb x^A.(\lb z^C.P)w}{\lb y^B.(\lb z^C.Q)w}\\
\to_{\beta}^2&\lb w^C.MD\pair{\lb x^{A}.[w/z]P}{\lb y^{B}.[w/z]Q}\\
=&\lb z^C.MD\pair{\lb x^{A}.P}{\lb y^{B}.Q}
\end{array}
$$

$$
\begin{array}{cl}
&M(\forall Y.D)\pair{\lb x^A\Lb Y.P}{\lb y^B\Lb Y.Q}\\
%=_{\alpha-ren}
=&M(\forall Y.D)\pair{\lb x^A\Lb Z.[Z/Y]P}{\lb y^B\Lb Z.[Z/Y]Q}\\
\to_{\atd}&\Lb Y.MD\pair{\lb x^A.(\Lb Z.[Z/Y]P)Y}{\lb y^B.(\Lb Z.[Z/Y]Q)Y}\\
\to_{\beta}^2&\Lb Y.MD\pair{\lb x^{A}.P}{\lb y^{B}.Q}
\end{array}
$$

$$
\begin{array}{cl}
&M(C_1\wedge C_2)\pair{\lb x^{A}.P}{\lb y^{B}.Q}\\
\leftarrow_{\eta}&M(C_1\wedge C_2)\pair{\lb x^{A}.\ipair i{Pi}}{\lb y^{B}.\ipair i{Qi}}\\
\to_{\delta}&\ipair i{ MC_i\pair{\lb x^{A}.Pi}{\lb y^{B}.Qi}}
\end{array}
$$

\end{proof}

According to the previous result, a $\delta$-reduction step can be broken into a $\atd\beta$-reduction sequence, but a $\atd$-reduction step can be derived only as a $\delta\eta$-equality. Given our insistence on reduction, rather than mere equality, in the main results to be shown below, the previous result is an argument to take the $\atd$-conversion rules as primitive, instead of the $\delta$-conversion rules.

The $\delta$-conversions pull down an introduction inference with which the two branches $P$ and $Q$ of a $\rpcase MxPyQC$ end. Dually, a commuting conversion pushes up to the two branches $P$ and $Q$ of a $\rpcase MxPyQC$ an elimination inference of which the mentioned $\rpcasesymb$ is main premiss.
%----------------------
\begin{defn}[Commuting conversion rules in $\f$]\label{def:epsilon}\quad
\begin{enumerate}
\item The $\ccd$-conversion rules are as follows:
$$
\begin{array}{rcl}
M(C_1\supset C_2)\pair{\lb x^{A}.P}{\lb y^{B}.Q}N & \to & MC_2\pair{\lb x^{A}.PN}{\lb y^{B}.QN}\\
M(C_1\wedge C_2)\pair{\lb x^{A}.P}{\lb y^{B}.Q}i & \to & MC_i\pair{\lb x^{A}.Pi}{\lb y^{B}.Qi}\\
M(\forall Y.C')\pair{\lb x^{A}.P}{\lb y^{B}.Q}C'' & \to & M([C''/Y]C')\pair{\lb x^{A}.PC''}{\lb y^{B}.QC''}
\end{array}
$$
%$$
%\begin{array}{rrcl}
%(\ccd)&\E\hole{MC\pair{\lb x^{A}.P}{\lb y^{B}.Q}}&\to&MD\pair{\lb x^{A}.\E\hole{P}}{\lb y^{B}.\E\hole{Q}}
%\end{array}
%$$
%\noindent where, for some $\Gamma$, $M$ has type $\forall X.((A_1\supset X)\wedge (A_2\supset X))\supset X$ in $\Gamma$, and %$C\vdash \E:B$.
%$\E$ has type $B$ and hole of type $C$ in $\Gamma$.

\item The $\cca$-conversion rules are as follows:
$$
\begin{array}{rcl}
M(C_1\supset C_2)N & \to & MC_2\\
M(C_1\wedge C_2)i & \to & MC_i\\
M(\forall Y.C')C'' & \to & M([C''/Y]C')
\end{array}$$
%$$
%\begin{array}{rrcl}
%(\cca)&\E\hole{MC}&\to&MD
%\end{array}
%$$
%\noindent where, for some $\Gamma$, $M$ has type $\forall X.X$ in $\Gamma$, and %$C\vdash \E:B$.
%$\E$ has type $B$ and hole of type $C$ in $\Gamma$.
\end{enumerate}
\end{defn}
%---------------------------------
\noindent We leave it to the reader to rewrite these rules in terms of $\rpcasesymb$ and $\rpabortsymb$.

Recall the discussion on how to define commutative conversions in $\ipc$. In $\f$, $\ccd$ is stricty contained in the following auxiliary rule:
$$(\ccd')\qquad\E\hole{\rpcase M{x^{A}}P{y^{B}}QC}\to\rpcase M{x^{A}}{\E\hole{P}}{y^{B}}{\E\hole{Q}}D\enspace.$$
The two rules are not the same due to the fact that in $\ccd'$ no connections is imposed on $C$ and $D$. If we wanted to constrain this rule to become equal to $\ccd$, we would define it only for certain pairs $(\E,C)$, with $D=D(\E,C)$ determined by $(\E,C)$ as follows:
\begin{equation}\label{eq:constraints-def-ccd}
\begin{array}{rcl}
D(\ehole N, E\supset F)&=&F\\
D(\ehole i,C_1\wedge C_2)&=&C_i \qquad\qquad(i=1,2)\\
D(\ehole E,\forall X.C_0)&=&[E/X]C_0
\end{array}
\end{equation}

Similar remarks apply to $\cca$ versus the auxiliary rule
$$(\cca')\qquad \E\hole{\rpabort MC}\to\rpabort MD\enspace.$$

Commutative conversions rules are named with symbols ``$\ccd$'' and ``$\cca$'', which are a variant of each other, again as a reminder that they express related commuting principles. That relation will be even more evident when reducing typable terms; but, as before with atomization and $\delta$, we postpone to Subsection \ref{subsec:properties-typed-atomization} all considerations about subject reduction and typing in connection with commutative conversions.
%The typographic distinction between ``$\ccd$'' and ``$\cca$'' is intended.

In all cases of the $\ccd$- and $\cca$-conversion rules, a universal instantiation with formula $C$ is replaced by another with some formula $D$, and in all cases $D$ is a sub-formula of $C$, except when $C=\forall Y.C'$ and $\E=\ehole C''$, for some formulas $C'$ and $C''$, in which case $D=[C''/Y]C'$. Let us compare $\atd$ with $\ccd$ in this situation:
$$
\begin{array}{rcl}
\rpcase MxPyQ{\forall Y.C'} & \to_{\atd} & \Lambda Y.\rpcase Mx{PY}y{QY}{C'}\\
\rpcase MxPyQ{\forall Y.C'}C'' & \to_{\ccd} & \rpcase Mx{PC''}y{QC''}{[C''/Y]C'}
\end{array}
$$
Starting from the $\ccd$-redex, an obvious alternative is to apply, not rule $\ccd$, but rule $\atd$ instead, to reach the intermediate term $(\Lambda Y.\rpcase Mx{PY}y{QY}{C'})C''$. The effect of $\ccd$-reduction is obtained by a further $\beta$-reduction step. In fact, each $\ccd$- and $\cca$-reduction step has a similar decomposition, as the first two items of the next result show.

%The next result shows that $\ccd$  and $\atd$ (respectively  $\cca$ and $\ata$) are related via $\beta\eta$-conversions.
%It also shows that the $\atd$-conversion rules (respectively the $\ata$-conversion rules) should be taken as primitive instead of the $\ccd$-conversion rules (respectively the $\cca$-conversion rules).

%------------------------------------------------------------
\begin{prop}[Atomization vs commuting conversion]\label{prop:atomization-vs-cc} Let $M,N\in \f$.
\begin{enumerate}
\item If $M\to_{\ccd}N$ then $M\to^2_{\atd \beta}N$.
\item If $M\to_{\cca}N$ then $M\to^2_{\ata \beta}N$.
\item If $M\to_{\atd}N$ then $M=_{\ccd\eta}N$.
\item If $M\to_{\ata}N$ then $M=_{\cca\eta}N$.
\end{enumerate}
\end{prop}
%----------------------------------------------

%----------------------------------------------
\begin{proof} We give four proofs by induction. We never show the inductive cases, which are routine.

The first assertion is proved by induction on $M\to_{\ccd}N$. There are three cases to consider, where the third corresponds to the discussion just before this propostition.
%according to $\E/C$.
%Case $\E=\ehole N$ and $C=C_1\supset C_2$ and $N:C_1$ and $B=C_2$.
$$
\begin{array}{rcl}
%LHS&=&
M(C_1\supset C_2)\pair{\lb x^{A}.P}{\lb y^{B}.Q}N %\\
&\to_{\atd}&(\lb z^{C_1}.MC_2\pair{\lb x^{A}.Pz}{\lb y^{B}.Qz})N\\
&\to_{\beta}&MC_2\pair{\lb x^{A}.PN}{\lb y^{B}.QN}
%&=&MC_2\pair{\lb x^{A}.\E\hole{P}}{\lb y^{B}.\E\hole{Q}}\\
%&=&RHS
\end{array}
$$
%Case $\E=\ehole i$ and $C=C_1\wedge C_2$ and $B=C_i$.
$$
\begin{array}{rcl}
%LHS&=&
M(C_1\wedge C_2)\pair{\lb x^{A}.P}{\lb y^{B}.Q}i %\\
&\to_{\atd}&(\ipair{j}{MC_j\pair{\lb x^{A}.Pj}{\lb y^{B}.Qj}})i\\
&\to_{\beta}&MC_i\pair{\lb x^{A}.Pi}{\lb y^{B}.Qi} %\\
%&=&MC_i\pair{\lb x^{A}.\E\hole{P}}{\lb y^{B}.\E\hole{Q}}\\
%&=&RHS
\end{array}
$$
%Case $\E=\ehole C''$ and $C=\forall Y.C'$ and $B=[C''/Y]C'$.
$$
\begin{array}{rcl}
%LHS&=&
M(\forall Y.C')\pair{\lb x^{A}.P}{\lb y^{B}.Q}C'' %\\
&\to_{\atd}&(\Lb Y.MC'\pair{\lb x^{A}.PY}{\lb y^{B}.QY})C''\\
&\to_{\beta}&M([C''/Y]C')\pair{\lb x^{A}.PC''}{\lb y^{B}.QC''} %\\
%&=&M([C''/Y]C')\pair{\lb x^{A}.\E\hole{P}}{\lb y^{B}.\E\hole{Q}}\\
%&=&RHS
\end{array}
$$

For the second assertion, we proceed by induction on $M\to_{\ata}N$. The three cases of the base are proved by simple calculations.
$$
\begin{array}{rcccl}
M(C_1\supset C_2)N&\to_{\ata}&(\lb z^{C_1}.MC_2)N&\to_{\beta}&MC_2\\
M(C_1\wedge C_2)i&\to_{\ata}&(\ipair{j}{MC_j})i&\to_{\beta}& MC_i\\
M(\forall Y.C')C''&\to_{\ata}&(\Lb Y.MC')C''&\to_{\beta}&M([C''/Y]C')
\end{array}
$$
%Case $\E=\ehole N$ and $C=C_1\supset C_2$ and $N:C_1$ and $B=C_2$.
%$$
%\begin{array}{rcl}
%LHS&=&M(C_1\supset C_2)N\\
%&\to_{\ata}&(\lb z^{C_1}.MC_2)N\\
%&\to_{\beta}&MC_2\\
%&=&RHS
%\end{array}
%$$
%Case $\E=\ehole i$ and $C=C_1\wedge C_2$ and $B=C_i$.
%$$
%\begin{array}{rcl}
%LHS&=&(M(C_1\wedge C_2))i\\
%&\to_{\ata}&(\ipair{j}{MC_j})i\\
%&\to_{\beta}&MC_i\\
%&=&RHS
%\end{array}
%$$
%Case $\E=\ehole C''$ and $C=\forall Y.C'$ and $B=[C''/Y]C'$.
%$$
%\begin{array}{rcl}
%LHS&=&M(\forall Y.C')C''\\
%&\to_{\ata}&(\Lb Y.MC')C''\\
%&\to_{\beta}&M([C''/Y]C')\\
%&=&RHS
%\end{array}
%$$

The third assertion is proved by induction on $M\to_{\atd} N$. There are three base cases.
%Case $C\supset D$.
$$
\begin{array}{rcl}
%LHS&=&
M(C_1\supset C_2)\pair{\lb x^{A}.P}{\lb y^{B}.Q} %\\
&\leftarrow_{\eta}&\lb z^{C_1}.(M(C_1\supset C_2)\pair{\lb x^{A}.P}{\lb y^{B}.Q})z\\
&\to_{\ccd}&\lb z^{C_1}.MC_2\pair{\lb x^{A}.Pz}{\lb y^{B}.Qz} %\\
%&=&RHS
\end{array}
$$
%Case $C_1\wedge C_2$.
$$
\begin{array}{rcl}
%LHS&=&
&&M(C_1\wedge C_2)\pair{\lb x^{A}.P}{\lb y^{B}.Q}\\
&\leftarrow_{\eta}&\pair {(M(C_1\wedge C_2)\pair{\lb x^{A}.P}{\lb y^{B}.Q})1}{(M(C_1\wedge C_2)\pair{\lb x^{A}.P}{\lb y^{B}.Q})2}\\
&\to^2_{\ccd}&\pair {MC_1\pair{\lb x^{A}.P1}{\lb y^{B}.Q1}}{MC_2\pair{\lb x^{A}.P2}{\lb y^{B}.Q2}}\\
&=&\ipair i{ MC_i\pair{\lb x^{A}.Pi}{\lb y^{B}.Qi}} %\\
%&=&RHS
\end{array}
$$
%Case $\forall Y.D$.
$$
\begin{array}{rll}
%LHS&=&
M(\forall Y.C_0)\pair{\lb x^{A}.P}{\lb y^{B}.Q} %\\
&\leftarrow_{\eta}&\Lb X.(M(\forall Y.C_0)\pair{\lb x^{A}.P}{\lb y^{B}.Q})X\\
&\to_{\ccd}&\Lb X.M([X/Y]C_0)\pair{\lb x^{A}.PX}{\lb y^{B}.QX}\\
%&=_{\alpha-ren}
&=&\Lb Y.MC_0\pair{\lb x^{A}.PY}{\lb y^{B}.QY} %\\
%&=&RHS
\end{array}
$$

For the fourth assertion, the calculations of the base of the proof by induction on $M\to_{\ata}N$ are as follows.
$$
\begin{array}{rcccl}
M(C\supset D)&\leftarrow_{\eta}&\lb z^C.(M(C\supset D))z&\to_{\cca}&\lb z^C.MD\\
M(C_1\wedge C_2)&\leftarrow_{\eta}&\pair {(M(C_1\wedge C_2))1}{(M(C_1\wedge C_2))2}&\to^2_{\cca}&\pair {MC_1}{MC_2}\\
M(\forall Y.C_0)&\leftarrow_{\eta}&\Lb X.(M(\forall Y.C_0))X&\to_{\cca}&\Lb X.M[X/Y]C_0
\end{array}
$$
Notice $\Lb X.M[X/Y]C_0=\Lb Y.MC_0$ as required.
%Case $C\supset D$.
%$$
%\begin{array}{rcl}
%LHS&=&M(C\supset D)\\
%&\leftarrow_{\eta}&\lb z^C.(M(C\supset D))z\\
%&\to_{\cca}&\lb z^C.MD\\
%&=&RHS
%\end{array}
%$$
%Case $C_1\wedge C_2$.
%$$
%\begin{array}{rcl}
%LHS&=&M(C_1\wedge C_2)\\
%&\leftarrow_{\eta}&\pair {(M(C_1\wedge C_2))1}{(M(C_1\wedge C_2))2}\\
%&\to^2_{\cca}&\pair {MC_1}{MC_2}\\
%&=&RHS
%\end{array}
%$$
%Case $\forall Y.D$.
%$$
%\begin{array}{rll}
%LHS&=&M(\forall Y.C')\\
%&\leftarrow_{\eta}&\Lb X.(M(\forall Y.C'))X\\
%&\to_{\cca}&\Lb X.M[X/Y]C'\\
%&=_{\alpha-ren}
%&=&\Lb Y.MC'\\
%&=&RHS
%\end{array}
%$$
\end{proof}
%------------------------------------------------------------

In the same way as Prop.~\ref{prop:variants-of-atomization} is an argument to take $\atd$-conversion rules as primitive, instead of the $\delta$-conversion rules, Prop.~\ref{prop:atomization-vs-cc} is an argument to take the $\atd$-conversion rules (respectively the $\ata$-conversion rules) as primitive instead of the $\ccd$-conversion rules (respectively the $\cca$-conversion rules). As a consequence, we may state below results in terms of $\ccd\cca\delta$-reduction, knowing that they may be immediately restated in terms of $\atd\ata$-reduction, in view of Props.~\ref{prop:variants-of-atomization} and \ref{prop:atomization-vs-cc} -- see for instance Theorem \ref{thm:strict-simulation} below.

%===================================================
\subsection{Strict simulation}\label{subsec:strict-simulation}

It has been observed \cite{FerreiraFerreira2009,TranchiniPistonePetrolo2019} that the Russell-Prawitz translation does not yield a simulation of proof reduction. Next we show that, once $\f$ is added $\atd\ata$-conversions, a simulation of proof reduction occurs. The simulation is even \emph{strict}, in the sense that each reduction step in $\ipc$ is mapped to a non-empty reduction sequence in the enriched system $\f$.

\begin{thm}[Strict simulation]\label{thm:strict-simulation}
If $M_1\to M_2$ in $\ipc$ then $\rpm {M_1}\to^+_{\beta\eta\ccd\cca\delta}\rpm {M_2}$ in $\f$ (hence $\rpm {M_1}\to^+_{\beta\eta\atd\ata}\rpm {M_2}$ in $\f$). More precisely:
\begin{itemize}
\item Case $R\in\{\betai,\etai,\betac,\etac\}$. If $M_1\to_R M_2$ in $\ipc$ then $\rpm{M_1}\to_R\rpm{M_2}$ in $\f$.
\item Case $R=\betad$. If $M_1\to_R M_2$ in $\ipc$ then $\rpm{M_1}\to_{\beta}^+\rpm{M_2}$ in $\f$.
\item Case $R=\etad$. If $M_1\to_R M_2$ in $\ipc$ then $\rpm{M_1}\to_{\eta\delta}^+\rpm{M_2}$ in $\f$.
\item Case $R\in\{\pii,\pic,\pid,\pia\}$. If $M_1\to_R M_2$ in $\ipc$ then $\rpm{M_1}\to_{\ccd}^+\rpm{M_2}$ in $\f$.
\item Case $R\in\{\abi,\abc,\abd,\aba\}$. If $M_1\to_R M_2$ in $\ipc$ then $\rpm{M_1}\to_{\cca}^+\rpm{M_2}$ in $\f$.
\end{itemize}
\end{thm}

\begin{proof}
For each rule $R$ of $\ipc$, one does an induction on $M_1\to_R M_2$. In each proof, the inductive cases follow routinely by induction hypothesis, since the various relations $\to_S$ and $\to^+_S$ in $\f$, with $S\in \{\beta, \eta, \ccd, \cca, \eta\delta\}$, are compatible; and the base case corresponds to the reduction rule $R$. The base cases relative to reduction rules pertaining to $\supset$ and $\wedge$ are trivial because $\rpm{(\cdot)}$ maps the constructions pertaining to these connectives in homomorphic fashion, and because $\rpm{([N/x]M)}=[\rpm N/x]\rpm M$. We detail the base cases relative to reduction rules pertaining to $\vee$ and $\perp$.

Case $\betad$: We prove that $\rpcase{\rpinjnsymb_i(N, A_1, A_2)}{x_1^{A_1}}{P_1}{x_2^{A_2}}{P_2}{C}\to^+_{\beta}[N/ x_i] P_i$ in $\f$.
$$
\begin{array}{cll}
&LHS&\\
=&(\Lb X.\lb w^{(A_1\supset X)\wedge(A_2\supset X)}.\proj i w N)C\pair{\lb x_1^{A_1}.P_1}{\lb x_2^{A_2}.P_2}&\\ %\textrm{(by def. of $\rpcasesymb$ and $\rpinjnsymb$)}\\
\to_{\betaall}&(\lb w^{(A_1\supset C)\wedge(A_2\supset C)}\proj iw N)\pair{\lb x_1^{A_1}.P_1}{\lb x_2^{A_2}.P_2}&\\
\to_{\betai}&\proj i{\pair{\lb x_1^{A_1}.P_1}{\lb x_2^{A_2}.P_2}}N&\\
\to_{\betac}&(\lb x_i^{A_i}.P_i)N&\\
\to_{\betai}&[N/x_i]P_i
\end{array}
$$
\noindent The first equality is justified by the definitions of $\rpcasesymb$ and $\rpinjnsymb$. To conclude the proof in this case, we need again the commutation of $\rpm{(\_)}$ with substitution.

Case $\etad$: We prove that $\rpcase M{x^A}{\rpinjn 1xAB}{y^B}{\rpinjn 2yAB}{A\dvee B}\to^+_{\delta\eta} M$ in $\f$.
$$
\begin{array}{cll}
&LHS&\\
=&M(A\dvee B)\pair{\lb x^A\Lb X\lb w^{(A\supset X)\wedge (B\supset X)}.\proj 1 w x}{\lb y^B\Lb X\lb w^{(A\supset X)\wedge (B\supset X)}.\proj 2 w y}\\
\to_{\delta}&\Lb X.M(((A\supset X)\wedge (B\supset X))\supset X)\pair{\lb x^A.\lb w.\proj 1 w x}{\lb y^B\lb w.\proj 2 w y}&\\
\to_{\delta}&\Lb X.\lb w.MX\pair{\lb x^A.{\proj 1 w x}}{\lb y^B.{\proj 2 w y}}&\\
\to_{\etai}^2&\Lb X.\lb w.MX\pair{{\proj 1 w}}{{\proj 2 w }}&\\
\to_{\etac}&\Lb X.\lb w.MXw&\\
\to_{\etai}&\Lb X.MX&\\
\to_{\etaall}&M&
\end{array}
$$
\noindent The first equality is justified by the definitions of $\rpcasesymb$ and $\rpinjnsymb$.

Cases $\pii$ and $\pic$ follow immediately from a single application of $\ccd$. Just notice that, in $\f$, $(\rpcase{M}{x^A}{P}{y^B}{Q}{C\supset D})N\to_{\ccd} \rpcase {M}{x^A}{PN}{y^B}{QN}{D}$ and $\proj i {\rpcase{M}{x^A}{P}{y^B}{Q}{C_1\wedge C_2}}\to_{\ccd} \rpcase {M}{x^A}{\proj i P}{y^B}{\proj i Q}{C_i}$.

Case $\pid$: We prove that, in $\f$,
$$
\begin{array}{l}\rpcase{\rpcase M{x_1^{A_1}}{P_1}{x_2^{A_2}}{P_2}{B_1\dvee B_2}}{y_1^{B_1}}{Q_1}{y_2^{B_2}}{Q_2}{C}\to_{\ccd}^+\\
\qquad\rpcase M{x_1^{A_1}}{\rpcase{P_1}{y_1^{B_1}}{Q_1}{y_2^{B_2}}{Q_2}C}{x_2^{A_2}}{\rpcase{P_2}{y_1^{B_1}}{Q_1}{y_2^{B_2}}{Q_2}C}C.
\end{array}$$
$$
\begin{array}{cll}
&LHS&\\
=&(M(B_1\dvee B_2)\pair{\lb x_1^{A_1}.P_1}{\lb x_2^{A_2}.P_2})C\pair {\lb y_1^{B_1}.Q_1}{\lb y_2^{B_2}.Q_2}\\
\to_{\ccd}&M(((B_1\supset C)\wedge (B_2\subset C))\supset C)\pair{\lb x_1^{A_1}.P_1C}{\lb x_2^{A_2}.P_2C}\pair {\lb y_1^{B_1}.Q_1}{\lb y_2^{B_2}.Q_2}&\\
\to_{\ccd}&MC\pair{\lb x_1^{A_1}.P_1C\pair {\lb y_1^{B_1}.Q_1}{\lb y_2^{B_2}.Q_2}}{\lb x_2^{A_2}.P_2C\pair {\lb y_1^{B_1}.Q_1}{\lb y_2^{B_2}.Q_2}}&\\
=&RHS&
\end{array}
$$

\noindent The definition of $\rpcasesymb$ justifies the equalities above.

Case $\pia$: In $\f$, we have that
$$
\aabbort{\rpcase M{x^A}P{y^B}Q{\dperp}}C\to_{\ccd}\rpcase M{x^A}{\aabbort PC}{y^B}{\aabbort QC}C\enspace.
$$
Indeed:
$$
LHS=(M\dperp\pair{\lb x^{A}.P}{\lb y^{B}.Q})C\to_{\ccd}MC\pair{\lb x^{A}.PC}{\lb y^{B}.QC}=RHS\enspace,
$$
where the two equalities are by definition of $\rpcasesymb$ and $\aabbortsymb$.
%$$
%\begin{array}{cll}
%&LHS&\\
%=&(M\dperp\pair{\lb x^{A}.P}{\lb y^{B}.Q})C&\textrm{(by def. of $\rpcasesymb$ and $\aabbortsymb$)}\\
%\to_{\ccd}&MC\pair{\lb x^{A}.PC}{\lb y^{B}.QC}&\\
%=&RHS&\textrm{(by def. of $\rpcasesymb$ and $\aabbortsymb$)}
%\end{array}
%$$

Cases $\abi$ and $\abc$ follow immediately from a single application of $\cca$. Just notice that $(\aabbort{M}{A\supset B})N\to_{\cca} \aabbort {M}{B}$ and $\proj i {\aabbort{M}{C_1\wedge C_2}}\to_{\cca} \aabbort {M}{C_i}$.

Case $\abd$: We prove that, in $\f$,
$$
\rpcase{\aabbort M{A\dvee B}}{x^A}P{y^B}QC\to_{\cca}^+\aabbort MC\enspace.
$$
$$
\begin{array}{cll}
&LHS&\\
=&M{(A\dvee B)}C\pair{\lb x^A.P}{\lb y^B.Q}&\textrm{(by def. of $\rpcasesymb$ and $\aabbortsymb$)}\\
\to_{\cca}&M((A\supset C)\wedge (B\supset C))\supset C)\pair{\lb x^A.P}{\lb y^B.Q}\\
\to_{\cca}&MC\\
=&RHS&\textrm{(by def. of $\aabbortsymb$)}
\end{array}
$$

Case $\aba$: We prove, in $\f$,
$$
\aabbort{\aabbort M{\dperp}}A\to_{\cca}\aabbort MA\enspace.
$$
Observe that
$$
LHS=M\dperp A\to_{\cca}MA=RHS\enspace,
$$
where the two equalities are by definition of $\aabbortsymb$.
%$$
%\begin{array}{cll}
%&LHS&\\
%=&M\dperp A &\textrm{(by def. of $\aabbortsymb$)}\\
%\to_{\cca}&MA&\\
%=&RHS&\textrm{(by def. of $\aabbortsymb$)}
%\end{array}
%$$
\end{proof}

Now suppose $M_1\to M_2$ in $\ipc$ and $M_1$ is typable in $\Gamma$. By subject reduction, $M_1$ and $M_2$ have the same type ($A$, say) in $\Gamma$, $M_1$ and $M_2$ can be seen as proofs of $A$, and the reduction step $M_1\to M_2$ can be seen as a normalization step between the proofs $M_1$ and $M_2$. Moreover, we know $\rpm{M_1}$ and $\rpm{M_2}$ have type $\rpm A$ in $\rpm\Gamma$, and the strict simulation theorem gives $\rpm {M_1}\to^+_{\beta\eta\atd\ata}\rpm {M_2}$ in $\f$: but does this reduction correspond to a sequence of normalization and atomization steps between successive proofs of $\rpm A$? To answer to this question, we have to investigate the typing of atomization (and $\delta$, and commutative) conversions.
%See the following subsection for an addendum to this theorem.

%==================================
\subsection{Properties of typable atomization}\label{subsec:properties-typed-atomization}

We investigate atomization in connection with typing. Recall for instance the $\atd$-conversion rule with $C=C_1\supset C_2$:
$$
M(C_1\supset C_2)\pair{\lb x^{A}.P}{\lb y^{B}.Q}\to\lb z^{C_1}.M C_2 \pair{\lb x^{A}.Pz}{\lb y^{B}.Qz} \qquad(*)
$$
From the assumption that the redex of $(*)$ has type $D$ in $\Gamma$, say, we \emph{cannot} infer that the contractum has the same type in $\Gamma$. That is, the subject-reduction property fails, if the conversion rule is formulated solely as $(*)$. The preservation of type from redex to contractum is guaranteed if, additionaly, we demand that $M$ has type $A\dvee B$ in $\Gamma$.

Now, the ammended rule is \emph{not} to take $(*)$ together with the requirement that, for \emph{some} $\Gamma$, the redex has a type and $M$ has type $A\dvee B$. Such rule (let alone its compatible closure) would still fail the subject-reduction test, because we could be given another $\Gamma'$ in which the redex had some type, and again, in the case $\Gamma'$ was not $\Gamma$, no guarantee would exist that the contractum had in $\Gamma'$ the same type as the redex.

The ammended rule is to take $(*)$ as defining a ternary relation, consisting of tuples $(N,N',\Gamma)$ where $N$ and $N'$ are, respectively, a redex and its contractum according to $(*)$, and $\Gamma$ is an environment in which $M$ has type $A\dvee B$. But, then, how to define the compatible closure of such ternary relation? We cannot simply close the pairs $(N,N')$ under the term-forming operations, because the $\Gamma$ may vary as we form new pairs - so the closure rules have to deal with $\Gamma$ as well. In the end, we have another set of tuples $(N,N',\Gamma)$, whose intuition is: $N\to N'$ is fine in $\Gamma$. The preservation of type from redex to contractum guaranteed above for the rule $(*)$ will hold now from $N$ to $N'$ - but only if the type is given in a fine $\Gamma$, not an arbitrary $\Gamma'$.

%----------------------
\begin{defn}[Fine Atomization]\label{def:fine-atomization} Let $\Gamma$ be an environment.
\begin{enumerate}
\item A $\atd$-redex $M C \pair{\lb x^A.P}{\lb y^B.Q}$ is \emph{fine in $\Gamma$} if $M$ has type $A \dvee B$ in $\Gamma$.
\item A $\ata$-redex $M C$ is \emph{fine in $\Gamma$} if $M$ has type $\dperp$ in $\Gamma$.
\item A root $\atd\ata$-reduction (an instance of the $\atd\ata$-conversion rules) is \emph{fine in $\Gamma$} if the redex is fine in $\Gamma$.
\item Let $R\in\{\atd,\ata,\atd\ata\}$. The fine root $R$-reductions define a ternary relation, namely
$$\mathcal{R}:=\{(M,M',\Gamma)|\textrm{$M\to M'$ is a root $R$-reduction fine in $\Gamma$}\}\enspace.$$
If $(M,M',\Gamma)\in\mathcal{R}$, we write ``\emph{$M\,R\,M'$ is fine in $\Gamma$}''. We now want to define the ``compatible closure'' of $\mathcal{R}$, that will be written ``\emph{$M\,\to_R\,M'$ is fine in $\Gamma$}'': it is another ternary relation defined inductively by closing $\mathcal{R}$ under the closure rules in Fig.~\ref{fig:fine-closure-rules}.
%Let $R':=\{(M,M')\in R|\exsts \Gamma\cdot(M,M',\Gamma)\in\mathcal{R}\}$.
\item $M$ is a \emph{fine $\atd\ata$-normal form in $\Gamma$} if $M\to_{\atd\ata}M'$ is fine in $\Gamma$ for no $M'$.
\item The closures
\begin{itemize}
\item ``\emph{$M\,\to_R^+\,M'$ is fine in $\Gamma$}'',
\item ``\emph{$M\,\to_R^*\,M'$ is fine in $\Gamma$}'' (fine $R$-reduction in $\Gamma$),
\item ``\emph{$M\,=_R\,M'$ is fine in $\Gamma$}'' (fine $R$-equality in $\Gamma$),
\end{itemize}of ``\emph{$M\,\to_R\,M'$ is fine in $\Gamma$}'' are obtained by closing the latter under the appropriate closure rules from the following list:
\begin{enumerate}
\item \textrm{$M \to M$ is fine in $\Gamma$} (fine reflexivity).
\item If \textrm{$M \to M'$ is fine in $\Gamma$}, then \textrm{$M' \to M$ is fine in $\Gamma$} (fine symmetry).
\item If \textrm{$M \to M'$ is fine in $\Gamma$} and \textrm{$M' \to M''$ is fine in $\Gamma$}, then \textrm{$M \to M''$ is fine in $\Gamma$} (fine transitivity).
\end{enumerate}
\end{enumerate}
\end{defn}
%--------------------------
\begin{figure}\caption{Fine compatible closure rules}\label{fig:fine-closure-rules}
\begin{enumerate}
\item[(i)] If $M\to M'$ is fine in $x:A,\Gamma$, then $\lb x^A.M\to\lb x^A.M'$ is fine in $\Gamma$.
\item[(ii)] If $M\to M'$ is fine in $\Gamma$, then $MN\to M'N$ is fine in $\Gamma$.
\item[(iii)] If $N\to N'$ is fine in $\Gamma$, then $MN\to MN'$ is fine in $\Gamma$.
\item[(iv)] If $M\to M'$ is fine in $\Gamma$, then $\pair MN\to \pair{M'}N$ is fine in $\Gamma$.
\item[(v)] If $N\to N'$ is fine in $\Gamma$, then $\pair MN\to \pair M{N'}$ is fine in $\Gamma$.
\item[(vi)] If $M\to M'$ is fine in $\Gamma$, then $\proj iM\to\proj i{M'}$ is fine in $\Gamma$.
\item[(vii)] If $M\to M'$ is fine in $\Gamma$, then $\Lambda X.M\to \Lambda X.{M'}$ is fine in $\Gamma$.
\item[(viii)] If $M\to M'$ is fine in $\Gamma$, then $MB\to{M'}B$ is fine in $\Gamma$.
\end{enumerate}
\end{figure}
%------------------------------
If we erase the $\Gamma$'s from the rules in Fig.~\ref{fig:fine-closure-rules}, then we obtain the ordinary closure rules defining the compatible closure. Therefore, if $M\to_R M'$ is fine in $\Gamma$, then $M\to_R M'$. In addition, we know that the $R$-redex contracted in this reduction step is fine in some $\Gamma'$ containing $\Gamma$ (the extra declarations in $\Gamma'$ are those relative to the $\lambda$-abstractions crossed when going from the root of $M$ to the contracted redex).

%Inspecting the conversion rules in items 1 and 2 of the previous definition, it is easy to see that every $\atd\ata$-contractum is typable in $\Gamma$, with the type of the redex. In fact, we have this peculiar form of subject reduction: if $M\to_{\atd\ata}M'$ in $\Gamma$, then $M$ and $M'$ have the same type in $\Gamma$.

%-----------------------
\begin{prop}[Fine subject reduction] \label{prop:fine-subject-reduction} If $M\to_{\atd\ata}M'$ is fine in $\Gamma$ and $M$ has type $A$ in $\Gamma$, then $M'$ has type $A$ in $\Gamma$.
\end{prop}
\begin{proof} By induction on $M\to_{\atd\ata}M'$ fine in $\Gamma$. For the base cases, going through the $\atd$-rules, we check that, if a $\atd$-redex $M C \pair{\lb x^A.P}{\lb y^B.Q}$ is fine and has a type in $\Gamma$, then that type is $C$, and $C$ is also the type of its contractum. Similarly for the $\ata$-rules. The inductive cases are routine.
\end{proof}

%-----------------------
\begin{prop}[Fine termination] \label{prop:fine-termination} Fine $\atd\ata$-reduction in $\Gamma$ starting from a typable term in $\Gamma$ is terminating.
\end{prop}
\begin{proof} First, given a type $C$, its size $|C|$ is defined by: $|X|=0$; $|A\supset B|=2|B|^2+3|B|+1$; $|A\wedge B|=1+|A|+|B|$; $|\forall X.A|=1+|A|$.

Next, we define a \emph{pre-redex} to be a term of the forms $MCQ$ or $MC$, where $C$ is not atomic. In the first case, the pre-redex is a generalization of a $\atd$-redex; and in the second case the pre-redex is exactly a $\ata$-redex. Pre-redexes are ranged over by $r$. We say a pre-redex $MCQ$ is \emph{fine in} $\Gamma$ if, for some types $A$ and $B$, $M$ has type $A\dvee B$ in $\Gamma$; and say the pre-redex $MC$ is \emph{fine in} $\Gamma$ if $M$ has type $\dperp$ in $\Gamma$. If the pre-redex $MCQ$ is fine in $\Gamma$, $MC$ is \emph{not} another pre-redex fine in $\Gamma$ (because the type of $M$ in $\Gamma$ is not $\dperp$).

Given $M$ typable in $\Gamma$, with type $A$, say, the unique typing derivation of $\Gamma\vdash M:A$ shows a given occurrence of a subterm $N$ of $M$ typable in some $\Gamma'$: such an occurrence is a \emph{pre-redex occurrence in $M$ (according to $\Gamma$)} if $N$ is a pre-redex fine in $\Gamma'$. Each pre-redex occurrence in $M$ is the occurrence of some pre-redex $r$ and has an associated environment $\Gamma'$. Different pre-redex occurrences in $M$ will be denoted $r_1$, $r_2$, etc. When we write $r_i$, we mean an occurrence of pre-redex $r$, and the related environment is denoted $\Gamma^{r_i}$. Let $\mathcal{R}$ be the set of pre-redex occurrences in $M$. Define
$$
%\begin{equation}\label{eq:measure}
W(M;\Gamma):=\sum_{r_i\in\mathcal{R}}w(r;\Gamma^{r_i})\enspace,
%\end{equation}
$$
where $w(r;\Gamma')$ is defined as follows:
\begin{itemize}
\item if $r=PCQ$, then $w(r; \Gamma')=|C|(1+W(P;\Gamma')+W(Q;\Gamma'))$;
\item if $r=PC$, then $w(r;\Gamma')=|C|(1+W(P;\Gamma'))$.
\end{itemize}

Now suppose $M$ is itself a pre-redex $r$ fine in $\Gamma$. Let us calculate $W(r;\Gamma)$:
\begin{itemize}
\item if $r=PCQ$, then $W(r;\Gamma)=w(r;\Gamma)+W(P;\Gamma)+W(Q;\Gamma)$, hence
$$W(r;\Gamma)=(|C|+1)(W(P;\Gamma)+W(Q;\Gamma))+|C|\qquad(*)$$
\item if $r=PC$, then $W(r;\Gamma)=w(r;\Gamma)+W(P;\Gamma)$, hence
$$W(r;\Gamma)=(|C|+1)(W(P;\Gamma))+|C|\qquad(**)$$
\end{itemize}

Notice how, in these calculations, we do not make use of the concrete definition of $|C|$. The same is true of the recursive definition of $W$ we give next, which relies on the previous calculations.

The easy cases read:
%$W(x)=0$; $W(\lb x.M)=W(M)$; $W\pair{M_1}{M_2}=W(M_1)+W(M_2)$; $W(\proj 1M)=W(\proj 2M)=W(M)$; and $W(\Lambda X.M)=W(M)$.
$$
\begin{array}{rcl}
W(x;\Gamma)&=&0\\
W(\lb x^A.M;\Gamma)&=&W(M;\Gamma,x:A)\\
W(\pair{M_1}{M_2};\Gamma)&=&W(M_1;\Gamma)+W(M_2;\Gamma)\\
W(\proj iM;\Gamma)&=&W(M;\Gamma)\\
W(\Lambda X.M;\Gamma)&=&W(M;\Gamma)
\end{array}
$$

As to $W(MC;\Gamma)$: if $MC$ is not a pre-redex fine in $\Gamma$, then $W(MC;\Gamma)=W(M;\Gamma)$; otherwise, $W(MC;\Gamma)=(|C|+1)W(M;\Gamma)+|C|$, due to $(**)$ above. Notice that, when $C$ is atomic (hence $MC$ is not a pre-redex and $|C|=0$), $(|C|+1)W(M;\Gamma)+|C|=W(M;\Gamma)$, so it does not matter which branch of the definition we use to calculate.

As to $W(MN;\Gamma)$: if $MN$ is not a pre-redex fine in $\Gamma$, then $W(MN;\Gamma)=W(M;\Gamma)+W(N;\Gamma)$; otherwise $MN=PCN$, $M=PC$ is not a pre-redex fine in $\Gamma$ (because the type of $P$ in $\Gamma$ is not $\dperp$), hence $W(M;\Gamma)=W(P;\Gamma)$, and therefore $W(MN;\Gamma)=(|C|+1)(W(M;\Gamma)+W(N;\Gamma))+|C|$, due to $(*)$ above. Notice that, when $M=PC$ with $C$ atomic (hence $MN=PCN$ is not a pre-redex and $|C|=0$), it does not matter again which branch of the definition we use to calculate, because $(|C|+1)(W(M;\Gamma)+W(N;\Gamma))+|C|=W(M;\Gamma)+W(N;\Gamma)$.

All is in place to prove:
$$
\textrm{If $M\to_{\atd\ata}N$ is fine in $\Gamma$, then $W(M;\Gamma)>W(N;\Gamma)$.}\qquad(***)
$$
Termination of fine $\atd\ata$-reduction in $\Gamma$ starting from $M$ follows from $(***)$ and the fact that $N$ is also typable in $\Gamma$ (due to fine subject reduction) .

The proof of $(***)$ is by induction on $M\to_{\atd\ata}N$ is in $\Gamma$, and we will make use of the recursive definition of $W$.

For the base cases, we have to check each conversion rule in Def.~\ref{def:atomization-rules}. The rules when $C=C_1\wedge C_2$ are challenging, because they cause duplication of terms. But, since measure $W$ is weighting pre-redexes, the $\atd$-rule when $C=C_1\supset C_2$ is challenging as well, since it generates new pre-redexes in some cases. The concrete definition of $|C|$ is important in proving the base cases.

Case $LHS:=M(C_1\supset C_2)\pair{\lb x^A.P}{\lb y^B.Q}\to_{\atd} \lb z^{C_1}MC_2\pair{\lb x^{A}.Pz}{\lb y^{B}.Qz}=:RHS$. Since $z\notin FV(M)$, $W(M;\Gamma,z:C_1)=W(M;\Gamma)$, so we just write $W(M)$. Similarly for $P$ and $Q$. The most favorable case, the case when $W(RHS)$ is smaller, is when neither $Pz$ nor $Qz$ is a new pre-redex. Then, $W(Pz)=W(P)$ and $W(Qz)=W(Q)$, and we calculate:
$$
\begin{array}{cl}
&W(LHS)\\
=&(|C_1\supset C_2|+1)(W(M)+W(P)+W(Q))+|C_1\supset C_2|\\
\geq&(|C_2|+1)(W(M)+W(P)+W(Q))+|C_1\supset C_2|\\
>&(|C_2|+1)(W(M)+W(P)+W(Q))+|C_2|\\
=&(|C_2|+1)(W(M)+W(Pz)+W(Qz))+|C_2|\\
=&W(RHS)
\end{array}
$$
We jump immediately to the less favorable case, when both $Pz$ and $Qz$ are pre-redexes. This means that $P=P'C_2$, $Q=Q'C_2$, $P$ and $Q$ have type $A'\dvee B'$, and $C_1=(A'\supset C_2)\wedge(B'\supset C_2)$, for some $P'$, $Q'$, $A'$ and $B'$. In this case, $W(P)=W(P')$, $W(Q)=W(Q')$, $W(Pz)=(|C_2|+1)W(P')+|C_2|$ and $W(Qz)=(|C_2|+1)W(Q')+|C_2|$. Then
$$
\begin{array}{cl}
&W(RHS)\\
=&(|C_2|+1)(W(M)+W(Pz)+W(Qz))+|C_2|\\
=&(|C_2|+1)W(M)+(|C_2|+1)^2W(P')+(|C_2|+1)^2W(Q')+2|C_2|^2+3|C_2|\\
\leq&(|C_2|+1)^2(W(M)+W(P')+W(Q'))+2|C_2|^2+3|C_2|\\
\leq&(|C_1\supset C_2|+1)(W(M)+W(P')+W(Q'))+2|C_2|^2+3|C_2|\\
=&(|C_1\supset C_2|+1)(W(M)+W(P)+W(Q))+2|C_2|^2+3|C_2|\\
<&(|C_1\supset C_2|+1)(W(M)+W(P)+W(Q))+2|C_2|^2+3|C_2|+1\\
=&(|C_1\supset C_2|+1)(W(M)+W(P)+W(Q))+|C_1\supset C_2|\\
=&W(LHS)
\end{array}
$$

Case $LHS:=M(C_1\wedge C_2)\pair{\lb x^A.P}{\lb y^B.Q}\to_{\atd}\ipair i{ MC_i\pair{\lb x^{A}.Pi}{\lb y^{B}.Qi}}=:RHS$. We calculate, omitting the fixed argument $\Gamma$ of $W$:
$$
\begin{array}{cl}
&W(LHS)\\
=&(|C_1|+|C_2|+2)(W(M)+W(P)+W(Q))+|C_1|+|C_2|+1\\
=&(|C_1|+1)(W(M)+W(P)+W(Q))+|C_1|+\\
 & + (|C_2|+1)(W(M)+W(P)+W(Q))+|C_2|+1\\
=&W(MC_1\pair{\lb x^A.\proj 1P}{\lb y^B.\proj 1Q}) + W(MC_2\pair{\lb x^A.\proj 2P}{\lb y^B.\proj 2Q}) +1\\
=&W(RHS)+1
\end{array}
$$
Case $LHS:=M(C_1\wedge C_2)\to_{\ata}\ipair i{ MC_i}=:RHS$. We calculate:
$$
\begin{array}{cl}
&W(LHS)\\
=&(|C_1|+|C_2|+2)W(M)+|C_1|+|C_2|+1\\
=&(|C_1|+1)W(M)+|C_1|+(|C_2|+1)W(M)+|C_2|+1\\
=&W(MC_1) + W(MC_2) +1\\
=&W(RHS)+1
\end{array}
$$

Case $M=\lb x^A.M'\to_{\atd\ata}\lb x^A.N'=N$, with $M'\to_{\atd\ata}N'$ fine in $\Gamma,x:A$. By IH, $W(M';\Gamma,x:A)>W(N';\Gamma,x:A)$. Hence $W(M;\Gamma)>W(N;\Gamma)$, because $W(M;\Gamma)=W(M';\Gamma,x:A)$ and $W(N;\Gamma)=W(N';\Gamma,x:A)$.

Equally easy are: the case $M=\Lambda X.M'\to_{\atd\ata}\Lambda X.N'=N$, with $M'\to_{\atd\ata}N'$ fine in $\Gamma$; the case $M=\pair{M_1}{M_2}\to_{\atd\ata}\pair{N_1}{M_2}=N$, with $M_1\to_{\atd\ata}N_1$ fine in $\Gamma$; the case $M=\pair{M_1}{M_2}\to_{\atd\ata}\pair{M_1}{N_2}=N$, with $M_2\to_{\atd\ata}N_2$ fine in $\Gamma$; and the case $M=\proj i{M'}\to_{\atd\ata}\proj i {N'}=N$, with $M'\to_{\atd\ata}N'$ fine in $\Gamma$. They all follow by IH and application of the recursive definition of $W$.

Case $M=M_1M_2\to_{\atd\ata}N_1M_2=N$ with $M_1\to_{\atd\ata}N_1$ fine in $\Gamma$. By IH, $W(M_1;\Gamma)>W(N_1;\Gamma)$. First, we prove this important remark: $M$ is a pre-redex fine in $\Gamma$ iff $N$ is a pre-redex fine in $\Gamma$. The ``only if'' implication is easy to see. The ``if'' implication is slightly more involved: if $N$ is a pre-redex, then $N_1=N_1'C$, for some $N_1',C$, and the circumstance $M_1\to_{\atd\ata}N_1$ implies that $M_1=M_1'C$ and $M_1'\to_{\atd\ata}N_1'$, since no root $\atd\ata$-reduction can produce the instantiation $N_1'C$. Both implications use fine subject reduction. Now that the remark is proved, we continue:
\begin{itemize}
\item First sub-case: $M$ is a pre-redex. Then $N$ is a pre-redex, $W(M;\Gamma)=(|C|+1)(W(M_1;\Gamma)+W(M_2;\Gamma))+|C|$, and $W(N;\Gamma)=(|C|+1)(W(N_1;\Gamma)+W(M_2;\Gamma))+|C|$. So $W(M;\Gamma)>W(N;\Gamma)$ follows from $W(M_1;\Gamma)>W(N_1;\Gamma)$.
\item Second sub-case: $M$ is not a pre-redex. Then $N$ is not a pre-redex, $W(M;\Gamma)=W(M_1;\Gamma)+W(M_2;\Gamma)$, and $W(N;\Gamma)=W(N_1;\Gamma)+W(M_2;\Gamma)$. So $W(M;\Gamma)>W(N;\Gamma)$ follows again from $W(M_1;\Gamma)>W(N_1;\Gamma)$.
\end{itemize}

The remaining two cases are proved similarly, each with the help of an ``important remark'', saying that $M$ is a pre-redex fine in $\Gamma$ iff $N$ is a pre-redex fine in $\Gamma$, and very easy to prove, using fine subject reduction. There is the case of $M=M_1M_2\to_{\atd\ata}M_1N_2=N$ with $M_2\to_{\atd\ata}N_2$ fine in $\Gamma$: its proof finished off very similarly to the previous case, by IH, the important remark and the recursive definition of $W$. The last case is $M=M'C\to_{\atd\ata}N'C=N$ with $M'\to_{\atd\ata}N'$ fine in $\Gamma$: its proof is finished off in the same way.
\end{proof}

%------------------------------
\begin{prop}[Atomic normal form] \label{prop:atomic-normal-form} If $M$ is typable in $\Gamma$, then $M$ has a unique fine $\atd\ata$-normal form in $\Gamma$ (which we call the \emph{atomic normal form} of $M$ in $\Gamma$).
\end{prop}
\begin{proof} Fine termination guarantees the existence of normal forms. As to uniqueness, we need to prove that fine $\atd\ata$-reduction in $\Gamma$ is confluent. By Newman's Lemma \cite{TS96}, it suffices to show local confluence. Suppose that $M\to_{\atd\ata}N_1$ and $M\to_{\atd\ata}N_2$ are fine in $\Gamma$. The two redexes contracted in these two steps can only overlap trivially (otherwise, in some $\Gamma'$, some $M'C$ would be a fine $\ata$-redex and, at the same time, part of a fine $\atd$-redex $M'CQ$, and hence $M'$ would have two different types in $\Gamma'$); but then it is easy to joint $N_1$ and $N_2$ in a common $\atd\ata$-reduct.
\end{proof}

This concludes the investigation of typable atomization. We turn briefly to $\delta$ and commutative conversions. The problem with subject reduction observed in the beginning of this subsection with $\atd$ is observed again with the $\delta$-rules in Def.~\ref{def:delta} and the $\ccd$-rules (resp. $\cca$-rules) in Def.~\ref{def:epsilon}. If we are given $\Gamma$ where the redex is typable, the corresponding contractum is typable only if, additionally, $M$ has type $A\dvee B$ (resp. type $\dperp$) in $\Gamma$. This determines what it means to be a $\delta$-, $\ccd$- or $\cca$-redex fine in $\Gamma$, and therefore what it means, for a root $\delta$-, $\ccd$- or $\cca$-reduction, to be fine in $\Gamma$.

A root $\beta\eta$-reduction is always fine in $\Gamma$. Hence, given $R$ any combination of $\beta,\eta,\delta,\ccd,\cca,\atd,\ata$, we know what it means to be a root $R$-reduction fine in $\Gamma$. Following the path in Def.~\ref{def:fine-atomization}, we define the versions of $\to_R$, $\to_R^+$, $\to_R^*$ and $=_R$ fine in $\Gamma$.

Now we have the tools to answer the question at the end of Subsection \ref{subsec:strict-simulation}, by providing an addendum to the strict simulation theorem.\\

\textbf{Theorem \ref{thm:strict-simulation} (full version):} \emph{If $M_1\to M_2$ in $\ipc$, then $\rpm {M_1}\to^+_{\beta\eta\ccd\cca\delta}\rpm {M_2}$ and $\rpm {M_1}\to^+_{\beta\eta\atd\ata}\rpm {M_2}$ in $\f$; moreover, these reductions are fine in $\rpm{\Gamma}$, if $M_1$ is typable in $\Gamma$.}
%if $M_1$ is typable in $\Gamma$, then $\rpm {M_1}\to^+_{\beta\eta\ccd\cca\delta}\rpm {M_2}$ and $\rpm {M_1}\to^+_{\beta\eta\atd\ata}\rpm {M_2}$ are fine in $\rpm{\Gamma}$.}

\begin{proof}
We have to go through all the cases in the proof of Theorem \ref{thm:strict-simulation} and check the extra statement. Let us start with the base cases. Again, reduction rules that just generate $\beta\eta$-steps in the target are no trouble, since $\beta\eta$-reduction is always fine in any environment. So we have to check rule $\etad$, the various $\pi$-rules, and the various $\ab$-rules. From the assumption that the redex is typable in $\Gamma$, a term $M$ has type $A\vee B$ or $\perp$ as expected, and so $\rpm M$ has type $\rpm{A}\dvee\rpm{B}$ or $\dperp$ as required to make all the $\delta$, $\ccd$ or $\cca$-steps in the shown simulation be fine in $\rpm\Gamma$. As to the inductive cases, they become routine, as soon as we realize that, for $S\in \{\beta, \eta, \ccd, \cca, \eta\delta\}$, the relations ``$N\to_S N'$ is fine in $\Delta$'' and ``$N\to^+_S N'$ is fine in $\Delta$'' not only enjoy the fine closure rules of Fig.~\ref{fig:fine-closure-rules}, but also another set of rules, the derivable rules that are the fine versions of those in Fig.~\ref{fig:compatibility-rules-F}, and which we refrain to write down.
\end{proof}

In a similar way we can go again through the proof of Prop.~\ref{prop:atomization-vs-cc}, to conclude, in particular: if $M\to_{\atd\ata}N$ is fine in $\Gamma$, then $M=_{\ccd\cca\eta}N$ is fine in $\Gamma$. We note a further consequence, to be used in the next subsection:
\begin{equation}\label{eq:fine-inclusion}
\textrm{If $M=_{\atd\ata}N$ is fine in $\Gamma$, then $M=_{\ccd\cca\eta}N$ is fine in $\Gamma$.}
\end{equation}

We finish this subsection with a final comments on $\ccd$ and $\cca$. Consider again rule $\ccd'$
$$
\E\hole{\rpcase M{x^{A}}P{y^{B}}QC}\to\rpcase M{x^{A}}{\E\hole{P}}{y^{B}}{\E\hole{Q}}D\enspace, %\qquad(*)
$$
and say this root reduction is \emph{fine in $\Gamma$} if $M$ has type $A\dvee B$ in $\Gamma$ and, furthermore, $\E$ has type $D$ and hole of type $C$ in $\Gamma$. Then, the connections between $C$ and $D$, spelled out in (\ref{eq:constraints-def-ccd}), are guaranteed by the typing rules for elimination contexts. So, such root reduction, when fine in $\Gamma$, is a root $\ccd$-reduction fine in $\Gamma$. The inverse is not necessarily true: in the case of $\E=\ehole N$, the typing rules give the bonus of $N$ typable in $\Gamma$, which is not necessarily the case in a fine root $\ccd$-reduction. The perfect match is obtained when the redex is typable: if $M$ is typable in $\Gamma$, then $M\to N$ is a fine root $\ccd$-reduction in $\Gamma$ iff $M\to N$ is a fine root $\ccd'$-reduction in $\Gamma$. Expressing $\ccd$ as the single rule $\ccd'$ will be used in the next subsection.

Similarly for root $\cca$-reduction and the $\cca'$-rule $\E\hole{\rpabort MC}\to\rpabort MD$.

%$$
%\begin{array}{rcl}
%\begin{eqnarray}
%\E\hole{\rpcase M{x^{A}}P{y^{B}}QC}&\to&\rpcase M{x^{A}}{\E\hole{P}}{y^{B}}{\E\hole{Q}}D\label{eq:ccd-in-F-with-context}\\
%\E\hole{\rpabort MC}&\to&\rpabort MD\label{eq:cca-in-F-with-context}
%\end{eqnarray}
%\end{array}
%$$

%===================================================
%\input{naturality}%%%%%%% Subsection{Connection with ``naturality'' conversions} in separate file

\subsection{Connection with dinaturality conversions}

In this subsection we prove that adding fine atomization conversions to system $\f$ does not cause the system to become inconsistent. 
%In view of Proposition \ref{prop:atomization-vs-cc}, 
In view of (\ref{eq:fine-inclusion}), it suffices to show the same for the addition of commutative conversions $\ccd$ and $\cca$. The strategy of the proof is the one that is implicit in \cite{TranchiniPistonePetrolo2019} in the informal justification that some general form of commutative conversions does not break consistency: one shows that the equality generated by adding the commutative conversions is contained in a bigger equality which is known to be consistent. Here, for the latter, we take the equality obtained by adding to system $\f$ a \emph{dinaturality conversion}, which we denote $\nu$.

The full formalization of $\nu$ would require bringing here the machinery of \cite{LatailladeLICS2009}. We refrain from doing that and just give a brief indication. Given formulas $A, C^+, C^-$ and a type variable $X$, the formula that results from substituting $C^+$ (resp. $C^-$) for the positive (resp. negative) occurrences of $X$ in $A$ is denoted $[(C^-,C^+)/X]A$. Notice $[(C,C)/X]A=[C/X]A$. Let $\mathcal{F}$ be the free category generated from system $\f$ (the ``syntactic'' category) by the usual method of categorical logic. Then, each $A$ determines a functor $[(\_,\_)/X]A:\mathcal{F}^{op}\times\mathcal{F}\to\mathcal{F}$. We refer the reader to \cite{LatailladeLICS2009} for the definition $[(f^-,f^+)/X]A$, for morphism $f^-, f^+$.

Let $x:A_1\vdash t:A_2$. The family of morphisms $[C/X]t:[(C,C)/X]A_1\to[(C,C)/X]A_2$ over $C$ is a \emph{dinatural transformation} between the functors $[(\_,\_)/X]A_1$ and $[(\_,\_)/X]A_2$ if, for every morphism $u:C\to D$, a certain diagram commutes, which means that certain two terms, determined by the given data, namely
$$
[(u,1_C)/X]A_1;[C/X]t;[(1_C,u)/X]A_2
$$
and
$$
[(1_D,u)/X]A_1;[D/X]t;[(u,1_D)/X]A_2\enspace,
$$
are $\beta\eta$-equal in system $\f$. In \cite{LatailladeLICS2009} a characterization is given of the terms $t$ which determine dinatural transformations. Conversion $\nu$ states that the referred diagram always commutes, that is, the above two terms are always equal. Hence, in system $\f$ plus $\nu$, every typable term determines a dinatural transformation.

We now show that $\ccd'\subseteq=_{\beta\eta\nu}$ (more precisely, a root $\ccd'$-reduction step fine in $\Gamma$ is contained in $=_{\beta\eta\nu}$). Let $M$ be such that $\Gamma\vdash M:A\dvee B$. Let $X$ be a type variable not free in $M$, and $P,P'$ be of type $C$, and $\Gamma|C\vdash \E:D$. For the purpose of matching the following calculation with the notation in the previous paragraphs, it is useful to put $t:=MX\pair{x}{x'}$. Then $\Gamma,x:A\supset X,x':B\supset X\vdash t:X$, so we may put $A_1:=(A\supset X)\wedge(B\supset X)$ and $A_2:=X$. Then:
$$
\begin{array}{rcl}
\E\hole{MC\pair{\lb z^A.P}{\lb {z'}^B.P'}}&=_{\eta}&\E\hole{MC\pair{\lb y.(\lb z^A.P)y}{\lb y'.(\lb {z'}^B.P')y'}}\\
&=&[\lb z.P/x][\lb z'.P'/x'](\E\hole{MC\pair{\lb y.xy}{\lb y'.x'y'}})\\
&=_{\nu}&[\lb z.P/x][\lb z'.P'/x'](MC\pair{\lb y.\E\hole{xy}}{\lb y'.\E\hole{x'y'}})\\
&=&MC\pair{\lb y.\E\hole{(\lb z.P)y}}{\lb y'.\E\hole{(\lb z'.P')y'}}\\
&=_{\beta}&MC\pair{\lb y.\E\hole{[y/z]P}}{\lb y'.\E\hole{[y'/z']P'}}\\
&=&MC\pair{\lb z.\E\hole{P}}{\lb z'.\E\hole{P'}}
\end{array}
$$

\noindent The $\nu$-conversion in this calculation is justified as follows. Before the conversion, we find the term $MC\pair{\lb y.xy}{\lb y'.x'y'}$, which is $[N/x][N'/x'][C/X]t$, where $N=\lb y.xy$ is $[(\E,1_C)/X](A\supset X)$ and $N'=\lb y'.x'y'$ is $[(\E,1_C)/X](B\supset X)$. Since $A_2=X$, $[(1_C,\E)/X]A_2=\E$. After the conversion, notice that $\lb y.\E\hole{xy}$ is $[(1_D,\E)/X](A\supset X)$ and $\lb y'.\E\hole{x'y'}$ is $[(1_D,\E)/X](B\supset X)$. Since $A_2=X$, $[(u,1_D)/X]A_2=1_D$.

As to $\cca$, it turns out that $\cca'\subseteq=_{\nu}$. Let $M$ be such that $\Gamma\vdash M:\dperp$. Let $X$ be a type variable not free in $M$, and $C\vdash \E:D$. For the purpose of recognizing a $\nu$-conversion, it is useful to put $t:=MX$ and allow a type $1$. Then $\Gamma,x:1\vdash t:X$, so we may put $A_1:=1$ and $A_2:=X$. From type $1$ we just require that $X$ does not occur in $1$, and that $[(f^-,f^+)/X]1$ is the identity $1_1$, that is $1$ as a trivial derivation. Then:
$$
\E\hole{MC}=_{\nu}MD\enspace.
$$
Indeed, $MC=[C/X]t$ and pre-composition with $[(\E,id_C)/]1$, since the latter is an identity; and $[(1_C,\E)/X]A_2=\E$, since $A_2=X$. After the conversion, $MD=[D/X]t$, and the pre-composition with $[(id_D,\E)/X]1$ has no effect since the latter is an identity; derivation $\E$ vanishes since $[(1_D,\E)/X]1$ is a trivial derivation; and $[(u,1_D)/X]A_2=1_D$, since $A_2=X$.

%------------------------------
\begin{thm}[Consistency]\label{thm:consistency}
%It is not the case that: $M=_{\beta\eta\atd\ata}N$ for all typable $M,N\in\f$.
In system $\f$ there are terms $M,N$ typable in $\Gamma$ such that $M=_{\beta\eta\atd\ata}N$ is not fine in $\Gamma$.
\end{thm}
%-----------------------------
\begin{proof} 
%Due to Proposition \ref{prop:atomization-vs-cc}, $=_{\beta\eta\atd\ata}\subseteq=_{\beta\eta\ccd\cca}$. By the calculations above, $=_{\beta\eta\ccd\cca}\subseteq=_{\beta\eta\nu}$. 
Having in mind (\ref{eq:fine-inclusion}) and the calculations above, we conclude: for all $M,N$ typable in $\Gamma$, if $M=_{\beta\eta\atd\ata}N$ is fine in $\Gamma$, then $M=_{\beta\eta\nu}N$. To complete the proof, we just invoke the consistency of $=_{\beta\eta\nu}$, which guarantees the existence of $M$ and $N$ typable in $\Gamma$ such that $M=_{\beta\eta\nu}N$ does not hold. The consistency of $=_{\beta\eta\nu}$, in turn, follows from the results in \cite{BainbridgeFreydScedrovScottTCS90}, where models of system $\f$ (with product types) are given %(in the category of partial equivalence relations over $\mathbb{N}$) 
such that every typable term in $\f$ is interpreted by a dinatural transformation, hence validating the dinatural conversion $\nu$.
\end{proof}
%%%%%%% Subsection{Connection with dinaturality conversions} in separate file  %------ contains an input of the file dinaturality.tex
%\section{Other embeddings and their comparison}\label{sec:comparison}
\section{Comparison of embeddings}\label{sec:comparison}

We recall the optimized translation of $\ipc$ into $\fat$, introduced by the authors in \cite{JESGFerreira2019}, and denoted $\am{(\cdot)}$. It comprises a translation of formulas, which is the same as in the Russell-Prawitz translation, so $\am A=\rpm A$; and comprises a translation of proof-terms (which induces a translation of derivations).

The translation of proof terms will rely on the following definition, taken from \cite{JESGFerreira2019}:

%-------------------------------------------------
\begin{defn}\label{def:admissible-constructions} In $\fat$:
\begin{enumerate}
\item Given $M,A,B$, given $i\in\{1,2\}$, we define
$$\ainjn iMAB:=\Lb X.\lb w^{(A\supset X)\wedge(B\supset X)}.\proj i w M\enspace,$$
where the bound variable $X$ is chosen so that $X\notin M,A,B$.
\item Given $M,P,Q,A,B,C$, we define $\acase M{x^A}P{y^B}QC$ by recursion on $C$ as follows:
$$
\begin{array}{rcl}
\acase M{x^A}P{y^B}QX&=&MX\pair{\lb x^A.P}{\lb y^B.Q}\\
\acase M{x^A}{P}{y^B}{Q}{C_1\wedge C_2}&=&\ipair i{\acase M{x^A}{\proj i{P}}{y^B}{\proj i{Q}}{C_i}}\\
%\acase MxPyQ{C_1\wedge C_2}&=&\langle\acase Mx{\proj 1P}y{\proj 1Q}{C_1},\\
%&&\,\,\,\acase Mx{\proj 2P}y{\proj 2Q}{C_2}\rangle\\
\acase M{x^A}P{y^B}Q{C\supset D}&=&\lb z^C.\acase M{x^A}{Pz}{y^B}{Qz}{D}\\
\acase M{x^A}P{y^B}Q{\forall X.C}&=&\Lb X.\acase M{x^A}{PX}{y^B}{QX}C
\end{array}
$$
where, in the third clause, the bound variable $z$ is chosen so that $z\neq x$, $z\neq y$ and $z\notin M,P,Q$; and in the fourth clause, the bound variable $X$ is chosen so that $X\notin M,P,Q,A,B$.
\item Given $M,A$, we define $\aabort MA$ by recursion on $A$ as follows:
$$
\begin{array}{rcl}
\aabort MX&=&MX\\
\aabort M{A_1\wedge A_2}&=&\pair{\aabort M{A_1}}{\aabort M{A_2}}\\
\aabort M{B\supset C}&=&\lb z^B.\aabort M{C}\\
\aabort M{\forall X.A}&=&\Lb X.\aabort MA
\end{array}
$$
where, in the third clause, the bound variable $z$ is chosen so that $z\notin M$; and in the fourth clause, the bound variable $X$ is chosen so that $X\notin M$.
\end{enumerate}
\end{defn}
%-------------------------------------------------

If we take the typing rules in Fig.~\ref{fig:admissible-typing-rules-F} and replace $\rpinjnsymb$, $\rpcasesymb$, and $\rpabortsymb$ by $\ainjnsymb$, $\acasesymb$, and $\aabortsymb$, respectively, we obtain admissible typing rules in $\fat$. Similarly, if we do the same replacements in Fig.~\ref{fig:compatibility-rules-F}, we obtain admissible compatibility rules in $\fat$. Such admissible rules in $\fat$ have been observed in \cite{JESGFerreira2019}.

%----------------------------
\begin{defn}
Given $M\in\ipc$, $\am M$ is defined by recursion on $M$ exactly as in Fig.~\ref{fig:translation-F}, except for the translation of $\injnsymb$, $\casesymb$ and $\abortsymb$, which now reads:
$$
\begin{array}{rcl}
\am{(\injn iMAB)}&=&\ainjn i{\am M}{\am A}{\am B}\qquad\textrm{($i=1,2$)}\\
\am{(\case M{x^A}P{y^B}QC)}&=&\acase{\am M}{x^{\am A}}{\am P}{y^{\am B}}{\am Q}{\am C}\\
\am{(\abort MA)}&=&\aabort{\am M}{\am A}
\end{array}
$$
%----------------------------------
\end{defn}

%------------------------------------------------
\begin{prop}[Type soundness] If $\Gamma\vdash M:A$ in $\ipc$, then $\am{\Gamma}\vdash\am M:\am A$ in $\fat$.
\end{prop}
%------------------------------------------------

To compare how the maps $\rpm{(\cdot)}$ and $\am{(\cdot)}$ translate proof terms amounts to compare $\rpinjnsymb$, $\rpcasesymb$, and $\rpabortsymb$, on the one hand, with $\ainjnsymb$, $\acasesymb$, and $\aabortsymb$, on the other hand. $\rpinjn iMAB$ and $\ainjn iMAB$ are defined in the same way, the repetition is due to stylistic reasons. The other comparisons use atomization conversions.

%------------------------------------------------------------

\begin{lem}\label{lem:comparison-case}
In $\f$: $\rpcase M{x^A}P{y^B}QC\to^*_{\atd}\acase  M{x^A}P{y^B}QC$; additionally, this reduction is fine in $\Gamma$, if $M$ has type $A\dvee B$ in $\Gamma$.
\end{lem}
%------------------------------------------------------------
\begin{proof}
The proof is by induction on $C$. We first check the first statement.

Case $C=X$. We calculate: $LHS = MX\pair{\lb x^A. P}{\lb y^B. Q} = RHS$, using twice the definition of $\rpcasesymb$.
%$$
%\begin{array}{cll}
%&LHS&\\
%=&MX\pair{\lb x^A. P}{\lb y^B. Q}&\textrm{(by def. of $\rpcasesymb$)}\\
%=&RHS&\textrm{(by def. of $\acasesymb$)}
%\end{array}
%$$

Case $C=C_1\supset C_2$.
$$
\begin{array}{rcll}
LHS&=&M(C_1\supset C_2)\pair{\lb x^A. P}{\lb y^B. Q}&\textrm{(by def. of $\rpcasesymb$)}\\
&\to_{\atd}&\lb z^{C_1}.MC_2\pair{\lb x^A.Pz}{\lb y^B .Qz} &\\
&=&\lb z^{C_1}.\rpcase M{x^A}{Pz}{y^B}{Qz}{C_2} &\\
&\to^*_{\atd}&\lb z^{C_1}.\acase M{x^A}{Pz}{y^B}{Qz}{C_2} &\textrm{(by IH)}\\
&=&RHS&\textrm{(by def. of $\acasesymb$)}
\end{array}
$$
Case $C=C_1\wedge C_2$.
$$
\begin{array}{rcll}
LHS&=&M(C_1\wedge C_2)\pair{\lb x^A. P}{\lb y^B. Q}&\textrm{(by def. of $\rpcasesymb$)}\\
&\to_{\atd}&\ipair i {MC_i\pair{\lb x^A.Pi}{\lb y^B .Qi}} &\\
&=&\ipair i {\rpcase M{x^A}{Pi}{y^B}{Qi}{C_i}} &\\
&\to^*_{\atd}&\ipair i {\acase M{x^A}{Pi}{y^B}{Qi}{C_i}} &\textrm{(by IH)}\\
&=&RHS&\textrm{(by def. of $\acasesymb$)}
\end{array}
$$
Case $C=\forall X. C_0$.
$$
\begin{array}{rcll}
LHS&=&M(\forall X. C_0)\pair{\lb x^A. P}{\lb y^B. Q}&\textrm{(by def. of $\rpcasesymb$)}\\
&\to_{\atd}&\Lb X. MC_0\pair{\lb x^A.PX}{\lb y^B .QX} &\\
&=&\Lb X. \rpcase M{x^A}{PX}{y^B}{QX}{C_0} &\\
&\to^*_{\atd}&\Lb X.\acase M{x^A}{PX}{y^B}{QX}{C_0} &\textrm{(by IH)}\\
&=&RHS&\textrm{(by def. of $\acasesymb$)}
\end{array}
$$

Now the additional statement. In the case $C=X$, there is nothing to check (by definition of fine reflexivity, $LHS\to^*_{\atd}LHS$ in any $\Gamma$). 

Case $C=C_1\supset C_2$. If $M$ has type $A\dvee B$ in $\Gamma$, the first $\atd$-reduction step in the calculation is, by definition, a fine root step in $\Gamma$. Given that $z\notin M$, $M$ has type $A\dvee B$ in $\Delta:=\Gamma,z:C_1$. By IH, the reduction under $\lb z$ is fine in $\Delta$.  Given that the relation ``$N\to^*_{\atd}N'$ is fine in $\Delta$'' enjoys the fine closure rules of Fig.~\ref{fig:fine-closure-rules}, we conclude that the second reduction in the calculation is fine in $\Gamma$.

The remaining cases are similar to, but simpler than this one.
\end{proof}

%------------------------------------------------------------
\begin{lem}\label{lem:comparison-abort}
In $\f$, $\rpabort MC\to^*_{\ata}\aabort  MC$; additionally, this reduction is fine in $\Gamma$, if $M$ has type $\dperp$ in $\Gamma$.
\end{lem}
%------------------------------------------------------------
\begin{proof}
The proof is by induction on $C$. We check first the first statement.

Case $C=X$. Then $LHS=MX=RHS$, using twice the definition of $\rpabortsymb$.
%$$
%\begin{array}{cll}
%&LHS&\\
%=&MX&\textrm{(by def. of $\rpabortsymb$)}\\
%=&RHS&\textrm{(by def. of $\aabortsymb$)}
%\end{array}
%$$

Case $C=C_1\supset C_2$.
$$
\begin{array}{rcll}
LHS&=&M(C_1\supset C_2)&\textrm{(by def. of $\rpabortsymb$)}\\
&\to_{\ata}&\lb z^{C_1}.MC_2 &\\
&=&\lb z^{C_1}.\rpabort M{C_2} &\\
&\to^*_{\ata}&\lb z^{C_1}.\aabort M{C_2} &\textrm{(by IH)}\\
&=&RHS&\textrm{(by def. of $\aabortsymb$)}
\end{array}
$$
Case $C=C_1\wedge C_2$.
$$
\begin{array}{rcll}
LHS&=&M(C_1\wedge C_2)&\textrm{(by def. of $\rpabortsymb$)}\\
&\to_{\ata}&\pair  {MC_1}{MC_2} &\\
&=&\pair  {\rpabort M{C_1}} {\rpabort M{C_2}} &\\
&\to^*_{\ata}&\pair {\aabort M{C_1}} {\aabort M{C_2}} &\textrm{(by IH)}\\
&=&RHS&\textrm{(by def. of $\aabortsymb$)}
\end{array}
$$
Case $C=\forall X. C_0$.
$$
\begin{array}{rcll}
LHS&=&M(\forall X. C_0)&\textrm{(by def. of $\rpabortsymb$)}\\
&\to_{\ata}&\Lb X. MC_0 &\\
&=&\Lb X. \rpabort M{C_0} &\\
&\to^*_{\ata}&\Lb X.\aabort M{C_0} &\textrm{(by IH)}\\
&=&RHS&\textrm{(by def. of $\acasesymb$)}
\end{array}
$$

The justification of the additional statement in each case of the proof is similar to what was done in the proof of the previous lemma. 
\end{proof}

%---------------------------------------------
\begin{prop}[Comparison of maps: proofs]\label{prop:comparison-of-maps}
For all $M\in\ipc$, $\rpm M\to^*_{\atd\ata}\am{M}$; additionally, this relation is fine in $\rpm\Gamma$, if $M$ is typable in $\Gamma$. Hence if $M$ is typable in $\Gamma$, $\am M$ is the atomic normal form of $\rpm M$ in $\rpm\Gamma$ (=$\am\Gamma$).
\end{prop}
%----------------------------------------------
\begin{proof}%The Russell-Prawitz translation  $\rpm{(\cdot)}$  and the $\am{(\cdot)}$-translation only differ in the translation of $\injnsymb$, $\casesymb$ and $\abortsymb$.
By induction on $M$. There are only two interesting cases, which follow by Lemmas \ref{lem:comparison-case} and \ref{lem:comparison-abort}. Notice $\am M$ is a fine $\atd\ata$-normal form, because $\am M\in\fat$.  Hence, by Proposition \ref{prop:atomic-normal-form}, $\am M$ is the unique $\atd\ata$-normal form of $\rpm M$ in $\rpm{\Gamma}$.
\end{proof}

We now want to compare how the maps $\rpm{(\cdot)}$ and $\am{(\cdot)}$ translate proof-reduction steps. Obviously, $R$-reduction steps, with $R\in\{\betai,\betac,\etai,\etac\}$, are translated in the same way by the two maps. Only the $R$-reduction steps, for $R$ a reduction rule pertaining to $\vee$ or $\perp$, are pertinent for the comparison, and so we concentrate on these from now on. The translation of such steps by $\rpm{(\cdot)}$ was detailed in Theorem \ref{thm:strict-simulation}, whereas the translation by $\am{(\cdot)}$ was detailed in \cite{JESGFerreira2019}. We want to see how the two pictures merge.

Let $R\in\{\betad,\pii,\pic,\abi,\abc,\abd,\aba\}$. A reduction step $M\to_R N$ in $\ipc$ gives rise to the diagram:
%---------
\newarrow{Many}----{>>}
\begin{equation}\label{eq:simple-picture}
\begin{diagram}[small]
M&&&&\am M&&\lMany^{\atd\ata}&&\rpm M\\
\dTo^R&&&&\dMany^{\beta\eta}&&&&\dMany_{\beta\eta\ccd\cca}\\
N&&&&\am N&&\lMany_{\atd\ata}&&\rpm N
\end{diagram}
\end{equation}
%--------
This follows from Theorem \ref{thm:strict-simulation} above, and also from Theorem 1 in \cite{JESGFerreira2019}, which guarantees $\am{M}\to^+_{\beta\eta}\am{N}$, whenever $M\to_R N$ in $\ipc$. The $\atd\ata$-reductions that bridges the two translations come from Proposition \ref{prop:comparison-of-maps}.

This picture has to be generalized, in order to accommodate the remaining cases $R\in\{\etad,\pid,\pia\}$. In these cases, the interaction between the terms translated with $\rpm{(\cdot)}$ and those translated with $\am{(\cdot)}$ will be richer than what can be expressed with Proposition \ref{prop:comparison-of-maps}. For this reason, we have to revisit Lemmas 6, 11 and 12 in \cite{JESGFerreira2019}, dedicated to $\etad$, $\pid$ and $\pia$, respectively. We do this next, but put the proofs in the appendix, since they are, to some extent, a repetition of the proofs already given in \cite{JESGFerreira2019}.
%In the following three lemmas, one finds diagrams where some $\beta$-reductions are marked as ``administrative''. See \cite{JESGFerreira2019} for a discussion of such reductions.

%------------------------------------------------
\begin{lem}[Rule $\etad$]\label{lem:eta-disjunction} Let $M\in\fat$ and $M'\in\f$ such that $M'\to_{\atd\ata}^*M$. Let
$$
\begin{array}{rcl}
LHS&=&\acase M{x^A}{\ainjn 1xAB}{y^B}{\ainjn 2yAB}{A\dvee B}\\
LHS'&=&\rpcase {M'}{x^A}{\ainjn 1xAB}{y^B}{\ainjn 2yAB}{A\dvee B}\\
RHS&=&M
\end{array}
$$
Then there is $Q\in\fat$ such that
\newarrow{Many}----{>>}
\begin{diagram}[small]
LHS&&\lMany^{\atd\ata}&&LHS'\\
&\rdMany_{\beta}&&\ldMany_{\delta\atd\ata}&\\
&&Q&&\\
&\ldMany_{\eta}&&&\\
RHS&&&&
\end{diagram}
\end{lem}
%------------------------------------------------
\begin{proof} Lemma 6 in \cite{JESGFerreira2019} just states $LHS\to_{\beta\eta}^+RHS$. The proof is a direct calculation. See the appendix for details.
\end{proof}

%------------------------------------------------
\begin{lem}[Rule $\pid$]\label{lem:pi-disjunction} Let $M,P_1,P_2,Q_1,Q_2\in\fat$ and $M',P'_1,P'_2,Q'_1,Q'_2\in\f$. Let
$$
\begin{array}{rcl}
LHS&=&\acase{\acase M{x_1}{P_1}{x_2}{P_2}{B_1\dvee B_2}}{y_1}{Q_1}{y_2}{Q_2}{C}\\
RHS&=&\acase M{x_1}{\acase{P_1}{y_1}{Q_1}{y_2}{Q_2}C}{x_2}{\acase{P_2}{y_1}{Q_1}{y_2}{Q_2}C}C\\
RHS'&=&\rpcase {M'}{x_1}{\rpcase{P'_1}{y_1}{Q'_1}{y_2}{Q'_2}C}{x_2}{\rpcase{P'_2}{y_1}{Q'_1}{y_2}{Q'_2}C}C
\end{array}$$
Suppose $T'\to_{\atd\ata}^*T$, for $T=M,P_1,P_2,Q_1,Q_2$. Then, there is $Q\in\fat$ such that
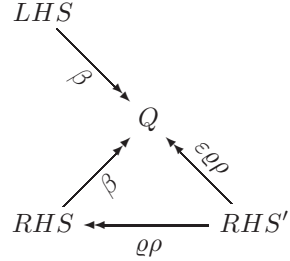
\newarrow{Many}----{>>}
\begin{diagram}[small]
LHS&&&&\\
&\rdMany_{\beta}&&&\\
&&Q&&\\
&\ruMany_{\beta}&&\luMany{\ccd\atd\ata}&\\
RHS&&\lMany_{\atd\ata}&&RHS'
\end{diagram}
\end{lem}
%------------------------------------------------
\begin{proof} For typographic reasons, we do not write the types of bound variables. Variables $y_1$ and $y_2$ have type $B_1$ and $B_2$, respectively. Variables $x_1$ and $x_2$ have type $A_1$ and $A_2$, where $A_1\dvee A_2$ is the type of $M$. These types stay unchanged throughout the proof.

Lemma 11 in \cite{JESGFerreira2019} just states $LHS=_{\beta}RHS$. The proof of the present lemma is by induction on $C$. See the appendix for details.
\end{proof}

%------------------------------------------------
\begin{lem}[Rule $\pia$]\label{lem:pi-absurdity} Let $M,P,Q\in\fat$ and $M',P',Q'\in\f$. Let
$$
\begin{array}{rcl}
LHS&=&\aabort{\acase M{x^A}P{y^B}Q{\dperp}}C\\
RHS&=&\qquad\acase M{x^A}{\aabort PC}{y^B}{\aabort QC}C\\
RHS'&=&\qquad\rpcase {M'}{x^A}{\rpabort {P'}C}{y^B}{\rpabort {Q'}C}C
\end{array}
$$
Suppose $T'\to_{\atd\ata}^*T$, for $T=M,P,Q$. Then, there is $Q\in\fat$ such that
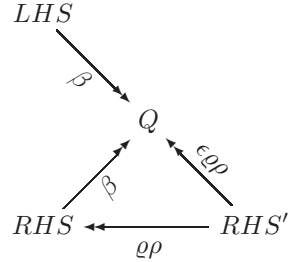
\newarrow{Many}----{>>}
\begin{diagram}[small]
LHS&&&&\\
&\rdMany_{\beta}&&&\\
&&Q&&\\
&\ruMany_{\beta}&&\luMany{\cca\atd\ata}&\\
RHS&&\lMany_{\atd\ata}&&RHS'
\end{diagram}
\end{lem}
%------------------------------------------------
\begin{proof} Lemma 12 in \cite{JESGFerreira2019} just states $LHS=_{\beta}RHS$. The proof is by induction on $C$. See the appendix for details.
\end{proof}

In order to avoid overloading too much the paper, we refrained from stating the full version of Lemmas \ref{lem:eta-disjunction}, \ref{lem:pi-disjunction}, and \ref{lem:pi-absurdity}. But the missing bits say the reduction from $LHS'$ or $RHS'$ to $Q$ is fine in $\Gamma$, provided $M$ has type $A\dvee B$ or $A_1\dvee A_2$ in $\Gamma$, and $P_1$ has type $B_1\dvee B_2$ or $\dperp$ in $\Gamma,x:A_1$ and similarly for $P_2$.

For each $R\in\{\etad,\pid,\pia\}$, we now show, using Lemmas \ref{lem:eta-disjunction}, \ref{lem:pi-disjunction} and \ref{lem:pi-absurdity}, a diagram in the style of (\ref{eq:simple-picture}), with a left half in $\ipc$ and a right half in $\f$. For instance, for $R=\etad$, the left half is $M\to_{\etad}N$ and the right half has the shape of the diagram in the statement of Lemma \ref{lem:eta-disjunction}, with $LHS$, $RHS$ and $LHS'$ replaced by $\am M$, $\am N$ and $\rpm M$, respectively. One should complete the diagram by adding $\rpm N$ (in the place of the missing $RHS'$) and drawing the reductions $\rpm N\to^*_{\atd\ata}\am N$ and $\rpm M\to^+_{\eta\delta}\rpm N$ (coming respectively from Proposition \ref{prop:comparison-of-maps} and Theorem \ref{thm:strict-simulation}). We do the same for $R=\pid$ and $R=\pia$, obtaining their respective diagrams.

One last thing. In the diagrams just obtained for $\etad$, $\pid$, and $\pia$, some reduction steps in the right half are \emph{administrative}, that is, they reduce redexes that do not correspond to redexes in the source terms $M,N\in\ipc$, but are redexes that were created by the translation $\am{(\_)}$ itself. This question was analyzed in detail in \cite{JESGFerreira2019}, specifically how some reduction steps in $\fat$ stated by Lemmas 6, 11 and 12 of \cite{JESGFerreira2019} can be classified as administrative, when they contribute to bridge $\am M$ and $\am N$. The analysis carries over to reduction steps in $\fat$ stated by Lemmas \ref{lem:eta-disjunction}, \ref{lem:pi-disjunction} and \ref{lem:pi-absurdity}, again when they contribute to bridge $\am M$ and $\am N$ - which is what happens in the diagrams just obtained for $\etad$, $\pid$, and $\pia$. That analysis allows us to say that: in the diagram for $\etad$, the $\beta$-reduction steps from $LHS=\am M$ to $Q$ are administrative; in the diagrams for $\pid$ and $\pia$, the $\beta$-reduction steps from $RHS=\am N$ to $Q$ are administrative.
%in the diagram for $\pia$, the $\beta$-reduction steps from $RHS=\am N$ to $Q$ are administrative.

We now have diagrams in the style of (\ref{eq:simple-picture}) for every reduction rule $R$ of $\ipc$ pertaining to disjunction or absurdity. The diagrams for $\etad$, $\pid$, and $\pia$ are slightly more complex, because they have a central term $Q$, and some reductions are classified as administrative. But we can define a general pattern that comprehends all of these diagrams, and thus explains the translation of any reduction steps $M\to_R N$ in $\ipc$: 
%with $R$ one such rule:

%, we see how to generalize picture (\ref{eq:simple-picture}) in order to comprehend all the reduction steps $M\to_R N$ in $\ipc$, for $R$ a reduction rule pertaining to disjunction or absurdity:% The result is in Fig.~\ref{fig:translation-reduction-step}.

%------------------
\begin{thm}[Comparison of maps: reduction]\label{thm:comparison-of-maps-reduction}
For $M\to_R N$ in $\ipc$, with $R$ a reduction rule pertaining to disjunction or absurdity, the reductions in Fig.~\ref{fig:translation-reduction-step} hold. Moreover, if $M$ is typable in $\Gamma$, then all reductions in Fig.~\ref{fig:translation-reduction-step} starting from $\rpm M$ or $\rpm N$ are fine in $\rpm\Gamma$.
\end{thm}
\begin{proof}
For the second statement, we have to invoke the subject reduction property of $\to_R$ in $\ipc$, the full version of Theorem \ref{thm:strict-simulation}, Proposition \ref{prop:comparison-of-maps}, and the full version of Lemmas \ref{lem:eta-disjunction}, \ref{lem:pi-disjunction}, and \ref{lem:pi-absurdity}.
\end{proof}

\begin{figure}\caption{Translation of a $R$-reduction step in $\ipc$, for $R$ a reduction rule pertaining to disjunction or absurdity. Terms $M$ and $N$ are in $\ipc$. Terms $\am M$, $\am N$, $Q_1$ and $Q_2$ are in $\fat$. Terms $\rpm M$ and $\rpm N$ are in $\f$. If $\am M=Q_1$, then the reduction $\rpm M\to_{\delta\atd\ata}^*Q_1$ is actually the reduction $\rpm M\to_{\atd\ata}^*\am M$. If $\am N=Q_2$, then the reduction $\rpm N\to_{\ccd\cca\atd\ata}^*Q_2$ is actually the reduction $\rpm N\to_{\atd\ata}^*\am N$. Notice that, due to Propositions \ref{prop:variants-of-atomization} and \ref{prop:atomization-vs-cc}, $\rpm M \to^*_{\beta\eta\rho\varrho} \rpm N$.}\label{fig:translation-reduction-step}
%---------------------------------
\newarrow{Many}----{>>}
\begin{diagram}
M&&\am M&&\lMany{\atd\ata}&&\rpm M\\
&&&\rdMany^{admin}_{\beta}&&\ldMany{\delta\atd\ata}&\\
&&&&Q_1&&\\
\dTo^R&&&&\dMany_{\beta\eta}&&\dMany_{\beta\eta\ccd\cca\delta}\\
&&&&Q_2&&\\
&&&\ruMany^{admin}_{\beta}&&\luMany{\ccd\cca\atd\ata}&\\
N&&\am N&&\lMany{\atd\ata}&&\rpm N
\end{diagram}
%---------------------------------
\end{figure} 
\section{Discussion}\label{sec:final}

We summarize our contribution. We proposed new conversions for system $\f$ whose purpose is to enforce atomic use of the universal instantiation. Such conversions explain the connection between the Russell-Prawitz translation and the translation into $\fat$ introduced by the authors \cite{JESGFerreira2019}, at the level of proofs (Proposition \ref{prop:comparison-of-maps}) and at the level of proof reduction (Theorem \ref{thm:comparison-of-maps-reduction}). In addition, only when system $\f$ is thus equipped does the Russell-Prawitz translation preserve proof reduction (Theorem \ref{thm:strict-simulation}) - and this without collapsing proof identity in system $\f$ (Theorem \ref{thm:consistency}), because the atomization conversions are not stronger than a certain ``dinaturality'' conversion known to preserve the consistency of equality.
%because the atomization conversions are very simple particular cases of the general commuting principle of \cite{TranchiniPistonePetrolo2019}.

Like the present paper, the recent article \cite{TranchiniPistonePetrolo2019} aims at finding new conversions for $\mathbf{NI}^2$ which allow to establish the preservation of proof identity by the Russell-Prawitz translation. In addition to the fact that we employ $\lambda$-terms, thereby making explicit the algorithmic aspect of the development, we see three main differences/improvements the present paper offers w.r.t.\ the work cited. First, we study the Russell-Prawitz embedding into $\f$ side-by-side with another embedding into $\fat$. This comprehensiveness is opportune because the translation into $\fat$ was perceived initially \cite{FerreiraFerreira2009} as a progress in the matter of preservation of proof identity, our Theorem \ref{thm:comparison-of-maps-reduction} bringing now a full clarification of the issue. Second, as opposed to the new conversion of \cite{TranchiniPistonePetrolo2019}, expressing ``naturality'' in the categorical sense, we propose a much simpler new conversion which, despite being connected to a very simple variant of the ``naturality'' conversion (namely conversions $\ccd$ and $\cca$), has a self-contained motivation (atomization of the uses of universal instantiation), and moreover not only delivers preservation of proof identity, but also makes a bridge between the Russell-Prawitz embedding and the embedding into $\fat$. Third, we obtained preservation of proof \emph{reduction} by the Russell-Prawitz embedding, while \cite{TranchiniPistonePetrolo2019} is only concerned with proof identity.

One wonders whether the results in \cite{TranchiniPistonePetrolo2019}, although stated in terms of proof identity, do establish (or could be modified to establish) results about proof reduction. But, with a single exception (Proposition 4.7. in \cite{TranchiniPistonePetrolo2019}), the answer is ``no'': (i) the results about ``m-closed'' instances of $\pid$ or $\etad$ rely essentially on an argument (see the proofs of Propositions 2.5 and 2.6 in \cite{TranchiniPistonePetrolo2019}) that starts with the $\beta$-normalization of a $\pid$ or $\etad$ \emph{contractum} - hence this $\beta$-normalization goes in the ``wrong direction'', does not preserve the direction of reduction; (ii) in the proof of Proposition 4.9 of \cite{TranchiniPistonePetrolo2019}), on preservation of $\etad$-equality, the Russell-Prawitz translation of the redex starts doing some steps of $\eta$-expansion, which again go in the ``wrong direction''.

Regarding the various embeddings of $\ipc$ into system $\fat$, it can be argued that the embedding $\am{(\_)}$ previously introduced by the authors \cite{JESGFerreira2019} has advantages over the original embedding based on instantiation overflow \cite{Ferreira2006,FerreiraFerreira2009}, in that shorter translations of proofs and of reduction sequences are obtained. In a recent paper in arXiv \cite{PTP-arXiv2019}, Pistone, Tranchini and Petrolo independently establish a connection between the Russell-Prawitz translation and yet another translation of $\ipc$ directly into system $\fat$, showing that they are equivalent modulo an extended equational theory for System $\f$. The translation into $\fat$ in \cite{PTP-arXiv2019} is not more ``economic'' than $\am{(\_)}$, but the exact comparison between the two deserve further investigation. However, the naturalness of the connection between $\am{(\_)}$ and the Russell-Prawitz translation established in the present paper, with $\am M$ being the $\atd\ata$-normal form of $\rpm{M}$, seems to the authors a strong indication of the special place occupied by the embedding $\am{(\_)}$ into system $\fat$.

%Indeed, given an $\ipc$ proof $M$, $\am M$ is the atomic normal form of $\rpm M$. So
Since $\am M$ is the atomic normal form of $\rpm M$, the embedding $\am{(\_)}$ makes full use of atomization at compile time. But, for the purpose of simulation, what one needs is a judicious use of atomization at run time. As observed in Fig.~\ref{fig:translation-reduction-step}, a reduction from $\am M$ to $\am N$ is still missing, for some cases of reduction $M\to N$ in $\ipc$, while a reduction always exists between the Russell-Prawitz translations $\rpm M$ and $\rpm N$. Such reduction sometimes contains atomization steps - those hidden in the $\ccd\cca\delta$-reduction steps pertaining to the reduction. So the simulation in system $\f$ by the Russell-Prawitz translation makes a controlled (not full) use of atomization depending on the source $\ipc$ reduction step $M\to N$, while such a resource is not available in system $\fat$, because in system $\fat$ we must stay fully atomized.

%Possible topic: the missing reduction $\am M\to^*\am N$ in . Linearization?

%%%%%%  ACKNOWLEDGEMENTS SECTION

%\section*{Funding}
\section*{Acknowledgements}

The authors thank the referees for their comments, in particular, for an error pointed out in the first version of the proof of termination of atomization. Both authors were supported by Funda\c{c}\~{a}o para a Ci\^{e}ncia e a Tecnologia [UIDB/00013/2020 and UIDP/00013/2020, UID/MAT/04561/2019, UID/CEC/00408/2019]. Both authors are grateful to Centro de Mate\-m\'{a}tica, Aplica\c{c}\~{o}es Fundamentais e Investiga\c{c}\~{a}o Operacional and the second author is also grateful to Large-Scale Informatics Systems Laboratory (Universidade de Lisboa).

\bigskip

\bibliography{bibrefs}

\begin{thebibliography}{10}

\bibitem{Aczel2001}
P.~Aczel.
\newblock The {R}ussell-{P}rawitz modality.
\newblock {\em Mathematical Structures in Computer Science}, 11(4):541--554,
  2001.

\bibitem{BainbridgeFreydScedrovScottTCS90}
E.~S. Bainbridge, P.~J. Freyd, A.~Scedrov, and P.~J. Scott.
\newblock Functorial polymorphism.
\newblock {\em Theor. Comput. Sci.}, 70(1):35--64, 1990.

\bibitem{JESGFerreira2019}
J.~{Esp\'irito Santo} and G.~Ferreira.
\newblock A refined interpretation of intuitionistic logic by means of atomic
  polymorphism.
\newblock {\em Studia Logica}, 2019.
\newblock https://doi.org/10.1007/s11225-019-09858-1.

\bibitem{Ferreira2006}
F.~Ferreira.
\newblock Comments on predicative logic.
\newblock {\em Journal of Philosophical Logic}, 35:1--8, 2006.

\bibitem{FerreiraFerreira2009}
F.~Ferreira and G.~Ferreira.
\newblock Commuting conversions vs. the standard conversions of the ``good''
  connectives.
\newblock {\em Studia Logica}, 92:63--84, 2009.

\bibitem{Ferreira2017}
G.~Ferreira.
\newblock Eta-conversions of $\ipc$ implemented in atomic $\f$.
\newblock {\em Logic Jnl IGPL}, 25(2):115--130, 2017.

\bibitem{GirardLafontTaylor89}
J-Y. Girard, Y.~Lafont, and P.~Taylor.
\newblock {\em Proofs and Types}.
\newblock Cambridge University Press, 1989.

\bibitem{LatailladeLICS2009}
J.~De Lataillade.
\newblock Dinatural terms in system {F}.
\newblock In {\em Proceedings of the 24th Annual {IEEE} Symposium on Logic in
  Computer Science, {LICS} 2009, 11-14 August 2009, Los Angeles, CA, {USA}},
  pages 267--276. {IEEE} Computer Society, 2009.

\bibitem{PTP-arXiv2019}
P.~Pistone, L.~Tranchini, and M.~Petrolo.
\newblock The naturality of natural deduction ({II}). {S}ome remarks on atomic
  polymorphism, 2019.
\newblock arXiv:1908.11353.

\bibitem{Prawitz65}
D.~Prawitz.
\newblock {\em Natural Deduction. A Proof-Theoretical Study}.
\newblock Almquist and Wiksell, Stockholm, 1965.

\bibitem{TranchiniPistonePetrolo2019}
L.~Tranchini, P.~Pistone, and M.~Petrolo.
\newblock The naturality of natural deduction.
\newblock {\em Studia Logica}, 107(1):195--231, 2019.

\bibitem{TS96}
A.~Troelstra and H.~Schwichtenberg.
\newblock {\em Basic Proof Theory}.
\newblock Cambridge University Press, 1996.

\end{thebibliography}
\bibliographystyle{plain}

%\newpage
\appendix
\section{Some proofs}\label{sec:some-proofs}

In this appendix we collect the proofs of Lemmas \ref{lem:eta-disjunction}, \ref{lem:pi-disjunction} and \ref{lem:pi-absurdity}.

%------------------------------------------------
\textbf{Lemma \ref{lem:eta-disjunction}.} Let $M\in\fat$ and $M'\in\f$ such that $M'\to_{\atd\ata}^*M$.
$$
\begin{array}{rcl}
LHS&=&\acase M{x^A}{\ainjn 1xAB}{y^B}{\ainjn 2yAB}{A\dvee B}\\
LHS'&=&\rpcase {M'}{x^A}{\ainjn 1xAB}{y^B}{\ainjn 2yAB}{A\dvee B}\\
RHS&=&M
\end{array}
$$
Then there is $Q\in\fat$ such that
\newarrow{Many}----{>>}
\begin{diagram}[small]
LHS&&\lMany^{\atd\ata}&&LHS'\\
&\rdMany_{\beta}^{admin}&&\ldMany_{\delta\atd\ata}&\\
&&Q&&\\
&\ldMany_{\eta}&&&\\
RHS&&&&
\end{diagram}
%------------------------------------------------
\begin{proof} $LHS'\to_{\atd\ata}^*LHS$ by $M'\to_{\atd\ata}^*M$ and Lemma \ref{lem:comparison-case}. $LHS$ is
$$
\Lb X.\acase{M}{x}{(\Lb Y\lb z.\proj 1 z x)X}{y}{(\Lb Y\lb z.\proj 2 z y)X}{((A\supset X)\wedge(B\supset X))\supset X}
$$
From the proof of Lemma 6 in \cite{JESGFerreira2019} we copy the following calculation, where we identify the term $Q$:
$$
\begin{array}{cll}
&LHS&\\
\to_{\betaall}^2&\Lb X.\acase{M}{x}{\lb z.\proj 1 z x}{y}{\lb z.\proj 2 z y}{((A\supset X)\wedge(B\supset X))\supset X}&\\
=&\Lb X\lb w.\acase{M}{x}{(\lb z.\proj 1 z x)w}{y}{(\lb z.\proj 2 z y)w}{X}&\\
\to_{\betai}^2&\Lb X\lb w.\acase{M}{x}{\proj 1 w x}{y}{\proj 2 w y}{X}&\\
=&\Lb X\lb w.{M}X\pair{\lb{x}.{\proj 1 w x}}{\lb{y}.{\proj 2 w y}}=:Q&\\
\to_{\etai}^2&\Lb X\lb w.{M}X\pair{{\proj 1 w}}{{\proj 2 w }}&\\
\to_{\etac}&\Lb X\lb w.{M}Xw&\\
\to_{\etai}&\Lb X.{M}X&\\
\to_{\etaall}&M&\\
=&RHS&
\end{array}
$$
We conclude as follows:
$$
\begin{array}{cll}
&LHS'&\\
=&M'(A\dvee B)\pair{\lb x^A\Lb X\lb w.w1x}{\lb y^B\Lb X\lb w.w2y}&\\
\to_{\atd\ata}^*&M(A\dvee B)\pair{\lb x^A\Lb X\lb w.w1x}{\lb y^B\Lb X\lb w.w2y}&\\
\to_{\delta}&\Lb X.M(((A\supset X)\wedge(B\supset X))\supset X)\pair{\lb x^A\lb w.w1x}{\lb y^B\lb w.w2y}&\\
\to_{\delta}&\Lb X\lb w.{M}X\pair{\lb{x}.{\proj 1 w x}}{\lb{y}.{\proj 2 w y}}&\\
=&Q&
\end{array}
$$
\end{proof}

%------------------------------------------------
\textbf{Lemma \ref{lem:pi-disjunction}.} Let $M,P_1,P_2,Q_1,Q_2\in\fat$ and $M',P'_1,P'_2,Q'_1,Q'_2\in\f$. Let
$$
\begin{array}{rcl}
LHS&=&\acase{\acase M{x_1^{A_1}}{P_1}{x_2^{A_2}}{P_2}{B_1\dvee B_2}}{y_1^{B_1}}{Q_1}{y_2^{B_2}}{Q_2}{C}\\
RHS&=&\acase M{x_1}{\acase{P_1}{y_1}{Q_1}{y_2}{Q_2}C}{x_2}{\acase{P_2}{y_1}{Q_1}{y_2}{Q_2}C}C\\
RHS'&=&\rpcase {M'}{x_1}{\rpcase{P'_1}{y_1}{Q'_1}{y_2}{Q'_2}C}{x_2}{\rpcase{P'_2}{y_1}{Q'_1}{y_2}{Q'_2}C}C
\end{array}$$
Suppose $T'\to_{\atd\ata}^*T$, for $T=M,P_1,P_2,Q_1,Q_2$. Then, there is $Q\in\fat$ such that
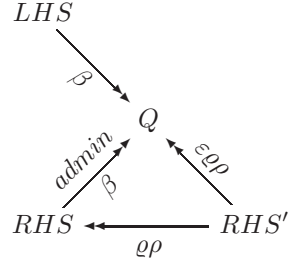
\newarrow{Many}----{>>}
\begin{diagram}[small]
LHS&&&&\\
&\rdMany_{\beta}&&&\\
&&Q&&\\
&\ruMany_{\beta}^{admin}&&\luMany{\ccd\atd\ata}&\\
RHS&&\lMany_{\atd\ata}&&RHS'
\end{diagram}

%------------------------------------------------
\begin{proof} $RHS'\to_{\atd\ata}^*RHS$ by the assumed reductions and Lemma \ref{lem:comparison-case}. The remainder of the diagram is proved by induction on $C$.

Case $C=Y$. $LHS$ is, by definition of $\acasesymb$,
$$
%(\Lb X.\lb w^{(B_1\supset X)\wedge(B_2\supset X)}.MX\pair{\lb x_1^{A_1}.P_1 Xw}{\lb x_2^{A_2}.P_2 Xw})Y\pair{\lb y_1^{\scriptstyle B_1}.Q_1}{\lb y_2^{\scriptstyle B_2}.Q_2}\enspace,
(\Lb X.\lb w^{(B_1\supset X)\wedge(B_2\supset X)}.MX\pair{\lb x_1.P_1 Xw}{\lb x_2.P_2 Xw})Y\pair{\lb y_1.Q_1}{\lb y_2.Q_2}\enspace,
$$
which, after one $\betaall$-reduction step, becomes
$$
(\lb w^{(B_1\supset Y)\wedge(B_2\supset Y)}.MY\pair{\lb x_1^{A_1}.P_1 Yw}{\lb x_2^{A_2}.P_2 Yw})\pair{\lb y_1^{B_1}.Q_1}{\lb y_2^{B_2}.Q_2}\enspace,
$$
because $X\notin M,P_1,P_2,A_1,A_2,B_1,B_2$. This term, in turn, yields, after one $\betai$-reduction step,
$$
MY\pair{\lb x_1^{A_1}.P_1 Y\pair{\lb y_1.Q_1}{\lb y_2.Q_2}}{\lb x_2^{A_2}.P_2 Y\pair{\lb y_1^{B_1}.Q_1}{\lb y_2^{B_2}.Q_2}}\enspace.
$$
This is $RHS$ by definition of $\acasesymb$. This calculation comes from the proof of Lemma 11 in \cite{JESGFerreira2019}. Now we add: put $Q:=RHS$. The reduction $RHS'\to_{\ccd\atd\ata}^*Q$ holds due to $RHS'\to_{\atd\ata}^*RHS$.

Case $C=C_1\supset C_2$. By definition of $\acasesymb$, $LHS$ is $\lb z^{C_1}.LHS_0$, where
$$LHS_0=\acase{\acase M{x_1}{P_1}{x_2}{P_2}{B_1\dvee B_2}}{y_1}{Q_1 z}{y_2}{Q_2 z}{C_2}\enspace.$$
On the other hand, $RHS$ is, by definition of $\acasesymb$,
$$\lb z^{C_1}.\acase M{x_1}{N_3}{x_2}{N_4}{C_2}\enspace,$$
with $N_3=(\acase{P_1}{y_1}{Q_1}{y_2}{Q_2}{C})z$, $N_4=(\acase{P_2}{y_1}{Q_1}{y_2}{Q_2}{C})z$. As argued in the proof of Lemma 11 in \cite{JESGFerreira2019}, $RHS$ does two administrative $\betai$-reduction steps (in the ``wrong'' direction), yielding $\lb z^{C_1}.RHS_0$, where
$$
%\begin{equation}\label{eq:middle-term}
RHS_0=\acase M{x_1}{N_1}{x_2}{N_2}{C_2}\enspace,
%\end{equation}
$$
with $N_1=\acase{P_1}{y_1}{Q_1 z}{y_2}{Q_2 z}{C_2}$, $N_2=\acase{P_2}{y_1}{Q_1 z}{y_2}{Q_2 z}{C_2}$.

Now $RHS'$ is the term
$$
M'(C_1\supset C_2)\pair{\lb x_1.P_1'(C_1\supset C_2)\pair{\lb y_1.Q_1'}{\lb y_2.Q_2'}}{\lb x_2.P_2'(C_1\supset C_2)\pair{\lb y_1.Q_1'}{\lb y_2.Q_2'}}
$$
which, after one $\atd$-reduction step, becomes
$$
\lb z^{C_1}.M' C_2\pair{\lb x_1.(P_1'(C_1\supset C_2)\pair{\lb y_1.Q_1'}{\lb y_2.Q_2'})z}{\lb x_2.(P_2'(C_1\supset C_2)\pair{\lb y_1.Q_1'}{\lb y_2.Q_2'})z}
$$
After two $\ccd$-reduction steps one obtains
$$
\lb z^{C_1}.M' C_2\pair{\lb x_1.P_1' C_2\pair{\lb y_1.Q_1'z}{\lb y_2.Q_2'z}}{\lb x_2.P_2' C_2\pair{\lb y_1.Q_1'z}{\lb y_2.Q_2'z}}
$$
The latter term is $\lb z^{C_1}.RHS'_0$, where $RHS'_0$ is
$$\rpcase {M'}{x_1}{\rpcase {P_1'}{y_1}{Q_1'z}{y_2}{Q_2'z}{C_2}}{x_2}{\rpcase {P_2'}{y_1}{Q_1'z}{y_2}{Q_2'z}{C_2}}{C_2}$$

By IH, applied to the terms $LHS_0$, $RHS_0$ and $RHS'_0$, one obtains a term $Q_0$ ``in the middle'' of three reduction sequences, as in the diagram above. The reduction relations involved are closed under the rule: $T\to T' \Rightarrow \lb z^{C_1}.T \to \lb z^{C_1}.T'$. So if we prefix the terms $LHS_0$, $RHS_0$, $RHS'_0$ and $Q_0$ with $\lb z^{C_1}$, the same reductions hold. We take $Q:=\lb z^{C_1}.Q_0$ and we are done.

Case $C=C_1\wedge C_2$. By definition of $\acasesymb$, $LHS$ is $\ipair i{LHS_{0i}}$, where
$$
LHS_{0i}=\acase{\acase M{x_1}{P_1}{x_2}{P_2}{B_1\dvee B_2}}{y_1}{\proj i{Q_1}}{y_2}{\proj i{Q_2}}{C_i}\enspace.
$$

On the other hand, $RHS$ is, by definition of $\acasesymb$,
$$
\ipair i{\acase M{x_1}{N_3}{x_2}{N_4}{C_i}}\enspace,
$$
with $N_3=\proj i{\acase{P_1}{y_1}{Q_1}{y_2}{Q_2}{C}}$ and $N_4=\proj i{\acase{P_2}{y_1}{Q_1}{y_2}{Q_2}{C}}$. As argued in the proof of Lemma 11 and in the comments on Theorem 1 in \cite{JESGFerreira2019}, $RHS$ does four administrative $\betac$-reduction steps (in the ``wrong'' direction), yielding the term $\ipair i{RHS_{0i}}$, where
$$
%\ipair i{\acase M{x_1}{\acase{P_1}{y_1}{\proj i{Q_1}}{y_2}{\proj i{Q_2}}{C_i}}{x_2}{\acase{P_2}{y_1}{\proj i{Q_1}}{y_2}{\proj i{Q_2}}{C_i}}{C_i}}\enspace.
RHS_{0i}=\acase M{x_1}{N_1}{x_2}{N_2}{C_i}\enspace,
$$
with $N_1=\acase{P_1}{y_1}{\proj i{Q_1}}{y_2}{\proj i{Q_2}}{C_i}$ and $N_2=\acase{P_2}{y_1}{\proj i{Q_1}}{y_2}{\proj i{Q_2}}{C_i}$.

Now $RHS'$ is the term
$$
M'(C_1\wedge C_2)\pair{\lb x_1.P_1'(C_1\wedge C_2)\pair{\lb y_1.Q_1'}{\lb y_2.Q_2'}}{\lb x_2.P_2'(C_1\wedge C_2)\pair{\lb y_1.Q_1'}{\lb y_2.Q_2'}}
$$
which, after one $\atd$-reduction step, becomes
$$
\ipair i{M'C_i\pair{\lb x_1.(P_1'(C_1\wedge C_2)\pair{\lb y_1.Q_1'}{\lb y_2.Q_2'})i}{\lb x_2.(P_2'(C_1\wedge C_2)\pair{\lb y_1.Q_1'}{\lb y_2.Q_2'})i}}
$$
After four $\ccd$-reduction steps, one obtains
$$
\ipair i{M'C_i\pair{\lb x_1.P_1'C_i\pair{\lb y_1.Q_1'i}{\lb y_2.Q_2'i}}{\lb x_2.P_2'C_i\pair{\lb y_1.Q_1'i}{\lb y_2.Q_2'i}}}
$$
The latter term is $\ipair i{RHS'_{0i}}$, where $RHS'_{0i}$ is
$$
\rpcase {M'}{x_1}{\rpcase{P_1'}{y_1}{Q_1'i}{y_2}{Q_2'i}{C_i}}{x_2}{\rpcase{P_2'}{y_1}{Q_1'i}{y_2}{Q_2'i}{C_i}}{C_i}
$$
For each $i=1,2$, and by IH, applied to the terms $LHS_{0i}$, $RHS_{0i}$ and $RHS'_{0i}$, one obtains a term $Q_{0i}$ ``in the middle'' of three reduction sequences, as in the diagram above. The reduction relations involved are closed under the rule: $T_1\to T_1'\textrm{ and }T_2\to T_2' \Rightarrow \pair{T_1}{T_2} \to \pair{T_1'}{T_2'}$. So if we form the pairs $\pair{LHS_{01}}{LHS_{02}}$, $\pair{RHS_{01}}{RHS_{02}}$, $\pair{RHS'_{01}}{RHS'_{02}}$ and $\pair{Q_{01}}{Q_{02}}$, the same reductions hold. We take $Q:=\pair{Q_{01}}{Q_{02}}$ and we are done.

Case $C=\forall Y.D$. By definition of $\acasesymb$, $LHS$ is $\Lambda Y.LHS_0$, where
$$
LHS_0=\acase{\acase M{x_1}{P_1}{x_2}{P_2}{B_1\dvee B_2}}{y_1}{Q_1Y}{y_2}{Q_2Y}{D}
$$
On the other hand, $RHS$ is, by definition of $\acasesymb$,
$$
\Lambda Y.\acase M{x_1}{N_3}{x_2}{N_4}{D}\enspace,
$$
with $N_3=(\acase{P_1}{y_1}{Q_1}{y_2}{Q_2}C)Y$, $N_4=(\acase{P_2}{y_1}{Q_1}{y_2}{Q_2}C)Y$. As argued in the proof of Lemma 11 in \cite{JESGFerreira2019}, $RHS$ does two administrative $\betaall$-reduction steps (in the ``wrong'' direction), yielding the term $\Lambda Y.RHS_0$, where
$$
RHS_0=\acase M{x_1}{N_1}{x_2}{N_2}{D}\enspace,
$$
with $N_1=\acase{P_1}{y_1}{Q_1Y}{y_2}{Q_2Y}D$, $N_2=\acase{P_2}{y_1}{Q_1Y}{y_2}{Q_2Y}D$.

Now $RHS'$ is
$$
M'(\forall Y.D)\pair{\lb x_1.P_1'(\forall Y.D)\pair{\lb y_1.Q_1'}{\lb y_2.Q_2'}}{\lb x_2.P_2'(\forall Y.D)\pair{\lb y_1.Q_1'}{\lb y_2.Q_2'}}
$$
which, after a $\atd$-reduction step, becomes
$$
\Lambda Y.M'D\pair{\lb x_1.(P_1'(\forall Y.D)\pair{\lb y_1.Q_1'}{\lb y_2.Q_2'})Y}{\lb x_2.(P_2'(\forall Y.D)\pair{\lb y_1.Q_1'}{\lb y_2.Q_2'})Y}
$$
After two $\ccd$-reduction steps, one obtains
$$
\Lambda Y.M'D\pair{\lb x_1.P_1'D\pair{\lb y_1.Q_1'Y}{\lb y_2.Q_2'Y}}{\lb x_2.P_2'D\pair{\lb y_1.Q_1'Y}{\lb y_2.Q_2'Y}}
$$
The latter term is $\Lambda Y.RHS'_0$, where $RHS'_0$ is
$$
\rpcase{M'}{x_1}{\rpcase{P_1'}{y_1}{Q_1'Y}{y_2}{Q_2'Y}D}{x_2}{\rpcase{P_2'}{y_1}{Q_1'Y}{y_2}{Q_2'Y}D}{D}
$$
By IH, applied to the terms $LHS_0$, $RHS_0$ and $RHS'_0$, one obtains a term $Q_0$ ``in the middle'' of three reduction sequences, as in the diagram above. The reduction relations involved are closed under the rule: $T\to T' \Rightarrow \Lambda Y.T \to \Lambda Y.T'$. So if we prefix the terms $LHS_0$, $RHS_0$, $RHS'_0$ and $Q_0$ with $\Lambda Y$, the same reductions hold. We take $Q:=\Lambda Y.Q_0$ and we are done.
\end{proof}

%------------------------------------------------
\textbf{Lemma \ref{lem:pi-absurdity}.} Let $M,P_1,P_2\in\fat$ and $M',P_1',P_2'\in\f$. Let
$$
\begin{array}{rcl}
LHS&=&\aabort{\acase M{x_1^{A_1}}{P_1}{x_2^{A_2}}{P_2}{\dperp}}C\\
RHS&=&\qquad\acase M{x_1^{A_1}}{\aabort {P_1}C}{x_2^{A_2}}{\aabort {P_2}C}C\\
RHS'&=&\qquad\rpcase {M'}{x_1^{A_1}}{\rpabort {P_1'}C}{x_2^{A_2}}{\rpabort {P_2'}C}C
\end{array}
$$
Suppose $T'\to_{\atd\ata}^*T$, for $T=M,P_1,P_2$. Then, there is $Q\in\fat$ such that
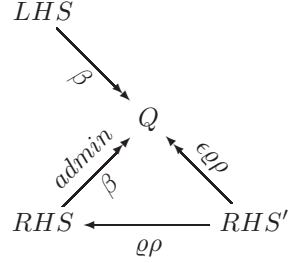
\newarrow{Many}----{>>}
\begin{diagram}[small]
LHS&&&&\\
&\rdMany_{\beta}&&&\\
&&Q&&\\
&\ruMany_{\beta}^{admin}&&\luMany{\cca\atd\ata}&\\
RHS&&\lTo_{\atd\ata}&&RHS'
\end{diagram}
%------------------------------------------------
\begin{proof} $RHS'\to_{\atd\ata}^*RHS$ by the assumed reductions and Lemmas \ref{lem:comparison-case} and \ref{lem:comparison-abort}. The remainder of the diagram is proved by induction on $C$.

Case $C=Y$. $LHS$ is, by definition of $\aabortsymb$ and $\acasesymb$,
$$
(\Lb X.MX\pair{\lb x_1^{A_1}.P_1X}{\lb x_2^{A_2}.P_2X})Y\enspace,
$$
which, after one $\betaall$-reduction step, becomes
$$
MY\pair{\lb x_1^{A_1}.P_1Y}{\lb x_2^{A_2}.P_2Y})\enspace,
$$
because $X\notin M,P_1,P_2,A_1,A_2$.

This is $RHS$ by definition of $\aabortsymb$ and $\acasesymb$. This calculation comes from the proof of Lemma 12 in \cite{JESGFerreira2019}. Now we add: put $Q:=RHS$. The reduction $RHS'\to_{\cca\atd\ata}^*Q$ holds due to $RHS'\to_{\atd\ata}^*RHS$.

Case $C=C_1\supset C_2$. By definition of $\aabortsymb$, $LHS$ is $\lb z^{C_1}.LHS_0$, where
$$LHS_0=\aabort{\acase M{x_1}{P_1}{x_2}{P_2}{\dperp}}{C_2}\enspace.$$
On the other hand, $RHS$ is, by definition of $\acasesymb$,
$$\lb z^{C_1}.\acase M{x_1}{\aabort {P_1}{C_1\supset C_2}z}{x_2}{\aabort {P_2}{C_1\supset C_2}z}{C_2}\enspace.$$ As argued in the proof of Lemma 12 and in the comments on Theorem 1 in \cite{JESGFerreira2019}, $RHS$ does two administrative $\betai$-reduction steps (in the ``wrong'' direction), yielding $\lb z^{C_1}.RHS_0$, where
$$
%\begin{equation}\label{eq:middle-term}
RHS_0=\acase M{x_1}{\aabort {P_1}{C_2}}{x_2}{\aabort {P_2}{C_2}}{C_2}\enspace.
%\end{equation}
$$

Now $RHS'$ is the term
$$
M'(C_1\supset C_2)\pair{\lb x_1.P_1'(C_1\supset C_2)}{\lb x_2.P_2'(C_1\supset C_2)}
$$
which, after one $\atd$-reduction step, becomes
$$
\lb z^{C_1}.M' C_2\pair{\lb x_1.P_1'(C_1\supset C_2)z}{\lb x_2.P_2'(C_1\supset C_2)z}.
$$
After two $\cca$-reduction steps one obtains
$$
\lb z^{C_1}.M' C_2\pair{\lb x_1.P_1' C_2}{\lb x_2.P_2' C_2}.
$$
The latter term is $\lb z^{C_1}.RHS'_0$, where $RHS'_0$ is
$$\rpcase {M'}{x_1}{\rpabort {P_1'}{C_2}}{x_2}{\rpabort {P_2'}{C_2}}{C_2}.$$

By IH, applied to the terms $LHS_0$, $RHS_0$ and $RHS'_0$, one obtains a term $Q_0$ ``in the middle'' of three reduction sequences, as in the diagram above. The reduction relations involved are closed under the rule: $T\to T' \Rightarrow \lb z^{C_1}.T \to \lb z^{C_1}.T'$. So if we prefix the terms $LHS_0$, $RHS_0$, $RHS'_0$ and $Q_0$ with $\lb z^{C_1}$, the same reductions hold. We take $Q:=\lb z^{C_1}.Q_0$ and we are done.

Case $C=C_1\wedge C_2$. By definition of $\aabortsymb$, $LHS$ is $\ipair i{LHS_{0i}}$, where
$$
LHS_{0i}=\aabort{\acase M{x_1}{P_1}{x_2}{P_2}{\dperp}}{C_i}\enspace.
$$

On the other hand, $RHS$ is, by definition of $\acasesymb$,
$$\ipair i{\acase M{x_1}{\proj i{\aabort {P_1}{C_1\wedge C_2}}}{x_2}{\proj i{\aabort {P_2}{C_1\wedge C_2}}}{C_i}}\enspace.$$ As argued in the proof of Lemma 12 and in the comments on Theorem 1 in \cite{JESGFerreira2019}, $RHS$ does two administrative $\betac$-reduction steps (in the ``wrong'' direction), yielding the term $\ipair i{RHS_{0i}}$, where
$$
%\ipair i{\acase M{x_1}{\acase{P_1}{y_1}{\proj i{Q_1}}{y_2}{\proj i{Q_2}}{C_i}}{x_2}{\acase{P_2}{y_1}{\proj i{Q_1}}{y_2}{\proj i{Q_2}}{C_i}}{C_i}}\enspace.
RHS_{0i}=\acase M{x_1}{\aabort {P_1}{C_i}}{x_2}{\aabort {P_2}{C_i}}{C_i}\enspace.
$$

Now $RHS'$ is the term
$$
M'(C_1\wedge C_2)\pair{\lb x_1.P_1'(C_1\wedge C_2)}{\lb x_2.P_2'(C_1\wedge C_2)}
$$
which, after one $\atd$-reduction step, becomes
$$
\ipair i{M'C_i\pair {\lb x_1.P_1'(C_1\wedge C_2)i}{\lb x_2.P_2'(C_1\wedge C_2)i}}.
$$
After two $\cca$-reduction steps one obtains
$$
\ipair i{M'C_i\pair {\lb x_1.P_1'C_i}{\lb x_2.P_2'C_i}}.
$$
The latter term is
$\ipair i{RHS'_{0i}}$, where $RHS'_{0i}$ is
$$
\rpcase {M'}{x_1}{\rpabort{P_1'}{C_i}}{x_2}{\rpabort{P_2'}{C_i}}{C_i}.
$$
For each $i=1,2$, and by IH, applied to the terms $LHS_{0i}$, $RHS_{0i}$ and $RHS'_{0i}$, one obtains a term $Q_{0i}$ ``in the middle'' of three reduction sequences, as in the diagram above. The reduction relations involved are closed under the rule: $T_1\to T_1'\textrm{ and }T_2\to T_2' \Rightarrow \pair{T_1}{T_2} \to \pair{T_1'}{T_2'}$. So if we form the pairs $\pair{LHS_{01}}{LHS_{02}}$, $\pair{RHS_{01}}{RHS_{02}}$, $\pair{RHS'_{01}}{RHS'_{02}}$ and $\pair{Q_{01}}{Q_{02}}$, the same reductions hold. We take $Q:=\pair{Q_{01}}{Q_{02}}$ and we are done.

Case $C=\forall Y.D$. By definition of $\aabortsymb$, $LHS$ is $\Lambda Y.LHS_0$, where
$$
LHS_0=\aabort{\acase M{x_1}{P_1}{x_2}{P_2}{\dperp}}{D}.
$$
On the other hand, $RHS$ is, by definition of $\acasesymb$,
$$
\Lambda Y.\acase M{x_1}{\aabort{P_1}{\forall Y.D}Y}{x_2}{\aabort{P_2}{\forall Y.D}Y}{D}\enspace.
$$
As argued in the proof of Lemma 12 in \cite{JESGFerreira2019}, $RHS$ does two administrative $\betaall$-reduction steps (in the ``wrong'' direction), yielding the term $\Lambda Y.RHS_0$, where
$$
RHS_0=\acase M{x_1}{\aabort{P_1}{D}}{x_2}{\aabort{P_2}{D}}{D}\enspace.
$$

Now $RHS'$ is
$$
M'(\forall Y.D)\pair{\lb x_1.P_1'(\forall Y.D)}{\lb x_2.P_2'(\forall Y.D)}
$$
which, after a $\atd$-reduction step, becomes
$$
\Lambda Y.M'D\pair{\lb x_1.(P_1'(\forall Y.D))Y}{\lb x_2.(P_2'(\forall Y.D))Y}.
$$
After two $\cca$-reduction steps, one obtains
$$
\Lambda Y.M'D\pair{\lb x_1.P_1'D}{\lb x_2.P_2'D}.
$$
The latter term is $\Lambda Y.RHS'_0$, where $RHS'_0$ is
$$
\rpcase{M'}{x_1}{\rpabort{P_1'}D}{x_2}{\rpabort{P_2'}D}{D}.
$$
By IH, applied to the terms $LHS_0$, $RHS_0$ and $RHS'_0$, one obtains a term $Q_0$ ``in the middle'' of three reduction sequences, as in the diagram above. The reduction relations involved are closed under the rule: $T\to T' \Rightarrow \Lambda Y.T \to \Lambda Y.T'$. So if we prefix the terms $LHS_0$, $RHS_0$, $RHS'_0$ and $Q_0$ with $\Lambda Y$, the same reductions hold. We take $Q:=\Lambda Y.Q_0$ and we are done.

\end{proof}

%\noindent\textbf{Affiliations:}\\

%\noindent Jos\'e Esp\'irito Santo\\
%Centro de Matem\'atica\\
%Universidade do Minho\\
%4710-057 Braga\\
%Portugal\\
%\texttt{jes@math.uminho.pt}\\

%\noindent Gilda Ferreira\\
%DCeT, Universidade Aberta, 1269-001 Lisboa, Portugal\\
%CMAFcIO, Universidade de Lisboa, 1749-016 Lisboa, Portugal\\
%\texttt{gmferreira@fc.ul.pt}

\end{document}